\documentclass[12pt]{amsart}
\usepackage{}

\usepackage{amsmath}
\usepackage{amsfonts}
\usepackage{amssymb}
\usepackage[all]{xy}           
\usepackage{bbm}
\usepackage{bbding}
\usepackage{txfonts}
\usepackage{amscd}
\usepackage{tikz}
\usetikzlibrary{matrix}
\usepackage[shortlabels]{enumitem}

\usepackage{ifpdf}
\ifpdf
\usepackage[colorlinks,final,backref=page,hyperindex]{hyperref}
\else
\usepackage[colorlinks,final,backref=page,hyperindex,hypertex]{hyperref}
\fi
\usepackage{tikz}
\usepackage[active]{srcltx}

\makeatletter

\newtheorem{thm}{Theorem}[section]
\newtheorem{lem}[thm]{Lemma}

\newtheorem{pro}[thm]{Proposition}
\newtheorem{ex}[thm]{Example}
\newtheorem{rmk}[thm]{Remark}
\newtheorem{defi}[thm]{Definition}

\newcommand {\emptycomment}[1]{}

\newcommand{\lon }{\,\rightarrow\,}
\newcommand{\be }{\begin{equation}}
\newcommand{\ee }{\end{equation}}

\newcommand{\g}{\mathfrak g}
\newcommand{\h}{\mathfrak h}

\newcommand{\huaS}{\mathcal{S}}

\newcommand{\huaG}{\mathcal{G}}

\newcommand{\huaC}{{\mathfrak{C}}}

\newcommand{\huaH}{\mathcal{H}}

\newcommand{\huaO}{{\mathcal{O}}}

\newcommand{\frkg}{\mathfrak g}

\newcommand{\half}{\frac{1}{2}}

\newcommand{\Courant}[1]{\left\llbracket  #1\right\rrbracket }

\newcommand{\jetd}{\mathbbm{d}}

\newcommand{\Id}{\rm{Id}}

\newcommand{\p}{\mathbbm{p}}
\newcommand{\br}[1]{   [ \cdot,    \cdot  ]   }
\newcommand{\id}{\mathbbm{i}}

\newcommand{\dM}{\mathrm{d}}

\newcommand{\Hom}{\mathrm{Hom}}

\newcommand{\gl}{\mathfrak {gl}}

\newcommand{\Ker}{\mathrm{Ker}}

\newcommand{\AD}{\mathfrak{ad}}

\newcommand{\ad}{\mathrm{ad}}

\newcommand{\K}{\mathbb{K}}

\newcommand{\Li}{\mathsf{3Lie}}
\newcommand{\cC}{{\rm c}\mathsf{3Lie}}
\begin{document}

\title[Deformations,  cohomologies and abelian extensions of compatible $3$-Lie
algebras]{Deformations,  cohomologies and abelian extensions of compatible $3$-Lie
algebras}

\author{Shuai Hou}
\address{Department of Mathematics, Jilin University, Changchun 130012, Jilin, China}
\email{houshuai19@mails.jlu.edu.cn}

\author{Yunhe Sheng}
\address{Department of Mathematics, Jilin University, Changchun 130012, Jilin, China}
\email{shengyh@jlu.edu.cn}

\author{Yanqiu Zhou}
\address{School of Science, Guangxi University of Science and Technology, Liuzhou 545006, China}
\email{zhouyanqiunihao@163.com}


\begin{abstract}
  In this paper, first we give the notion of a compatible $3$-Lie algebra and construct a  bidifferential graded
Lie algebra whose Maurer-Cartan elements are compatible $3$-Lie algebras. We also obtain the bidifferential graded
Lie algebra that governs deformations of a compatible $3$-Lie algebra. Then we
introduce a cohomology theory of a compatible $3$-Lie algebra with coefficients in itself and show that there is a one-to-one correspondence between equivalent classes of infinitesimal deformations of a compatible $3$-Lie algebra and the second cohomology group. We further study 2-order 1-parameter deformations of a compatible $3$-Lie algebra  and introduce
the notion of a Nijenhuis operator on a compatible $3$-Lie algebra, which could give rise to a trivial deformation.
At last, we introduce a  cohomology theory   of a compatible $3$-Lie algebra  with coefficients in arbitrary representation and classify abelian extensions of a compatible $3$-Lie algebra  using the
second cohomology group.

\end{abstract}



\maketitle

\tableofcontents

\allowdisplaybreaks


\section{Introduction}\label{sec:intr}


Compatible algebraic structures refer to two algebraic structures of the same kind in a linear category such that any linear combination of multiplications corresponding to these two algebraic structures   still defines the same kind of algebraic structure. Compatible algebraic structures have been widely studied in mathematics and mathematical  physics. Golubchik and Sokolov   studied compatible Lie algebras with the background of integrable equations \cite{GS1}, classical Yang-Baxter equations \cite{GS2}, loop algebras over
Lie algebras \cite{GS3}.  Compatible Lie algebras are also related to  elliptic theta functions \cite{OS1}. See   \cite{Pan,S,WB} for more details about classification, operads and bialgebra theory of compatible Lie algebras. Recently, the theory of bidifferential graded Lie algebras was developed in \cite{LSB}, by which  the deformation
theory and the cohomology theory of compatible Lie algebras were established. Then  compatible $L_{\infty}$-algebras
and compatible associative algebras were studied in \cite{DAS,DAS-1} using similar ideas.

The notion of $3$-Lie algebras and more generally, $n$-Lie algebras was introduced in~\cite{Filippov}, and have been
widely studied from different aspects. The $n$-Lie algebra is the algebraic structure corresponding to Nambu mechanics \cite{N}.
$3$-Lie algebras play important roles in the study of the Bagger-Lambert-Gustavsson theory of multiple
M2-branes \cite{M2branes,MFM1}. Cohomology and
deformation  theories of $3$-Lie algebras have been deeply studied in \cite{ABM,Izquierdo,deformation,NR bracket of n-Lie,Takhtajan1,Zhangtao}. See the review article \cite{review,A. Makhlouf} for more details.

The purpose of the paper is to study compatible $3$-Lie algebras. We construct a bidifferential graded Lie algebra such that
compatible $3$-Lie algebras can be characterized by its Maurer-Cartan elements. We introduce a cohomology
theory of a compatible $3$-Lie algebra with coefficients in itself with the help of the aforementioned bidifferential graded Lie algebra.
As applications, we study infinitesimal deformations and $2$-order $1$-parameter deformations of a compatible $3$-Lie algebra respectively. We show that infinitesimal deformations of a compatible $3$-Lie algebra are classified by the second cohomology group. We introduce the notion of a Nijenhuis operator on a compatible 3-Lie algebra and show that a Nijenhuis operator gives rise to a trivial $2$-order $1$-parameter deformation of the compatible $3$-Lie algebra. Moreover, we introduce the cohomology of a compatible $3$-Lie algebra with coefficients in arbitrary
representation, and classify abelian
extensions of a compatible $3$-Lie algebra   by the second cohomology group.

As aforementioned, compatible Lie algebras and 3-Lie algebras both have various applications in mathematical physics, so compatible 3-Lie algebras also have potential applications in mathematical physics, which will be studied in the future. Note that the theory of compatible 3-Lie algebras is not totally parallel to that of compatible Lie algebras, e.g. a Nijenhuis operator on a Lie algebra $(\g,[\cdot,\cdot])$ gives rise to a deformed Lie algebra $(\g,[\cdot,\cdot]_N)$ such that $(\g,[\cdot,\cdot],[\cdot,\cdot]_N)$ is naturally a compatible Lie algebra, while this is not true for 3-Lie algebras. A Nijenhuis operator on a 3-Lie algebra $(\g,[\cdot,\cdot,\cdot])$ also gives rise to a deformed 3-Lie algebra $(\g,[\cdot,\cdot,\cdot]_N)$. But  $(\g,[\cdot,\cdot,\cdot],[\cdot,\cdot,\cdot]_N)$ is in general not a compatible 3-Lie algebra unless extra conditions are satisfied (see Proposition \ref{pro:c3c}).

The paper is organized as follows. In Section \ref{sec:cohomology I}, we introduce the notion of a compatible $3$-Lie algebra, and construct a bidifferential graded Lie algebra
whose  Maurer-Cartan elements are compatible 3-Lie algebra structures. We further construct the  bidifferential graded Lie algebra
governing deformations of a compatible 3-Lie algebra.
In Section \ref{sec:coh}, we introduce the cohomology theory of a compatible $3$-Lie algebra  with coefficients in itself and study infinitesimal deformations.
We also study 2-order 1-parameter deformations of a compatible $3$-Lie algebra and introduce the notion of
a Nijenhuis operator on a compatible $3$-Lie algebra, which could give rise to a trivial deformation.
In Section \ref{sec:cohomology II},
we introduce a cohomology theory of a compatible $3$-Lie algebra with coefficients in arbitrary representation and use the second cohomology group to classify abelian extensions of a compatible $3$-Lie algebra.

\vspace{2mm}
In this paper, all the vector spaces are over algebraically closed field $\mathbb K$ of characteristic $0$, and finite dimensional.
\vspace{2mm}

{\bf Acknowledgements.} This research was  supported by NSFC (11922110).

\section{The bidifferential gLa governing deformations of a compatible $3$-Lie algebra} \label{sec:cohomology I}

In this section, first we construct the bidifferential graded Lie
algebra (gLa) whose Maurer-Cartan elements are compatible $3$-Lie algebra
structures. Then we give the bidifferential graded Lie algebra  governing deformations of a compatible $3$-Lie algebra.
\begin{defi}{\rm (\cite{Filippov})}\label{defi:3Lie}
A {\bf 3-Lie algebra}
is a vector space $\g$ together with a skew-symmetric linear map $[\cdot,\cdot,\cdot]_{\g}:\wedge^{3}\g\rightarrow \g$, such that for $ x_{i}\in \g, 1\leq i\leq 5$, the following {\bf Fundamental Identity} holds:
\begin{eqnarray}
\nonumber\qquad &&[x_1,x_2,[x_3,x_4, x_5]_{\g}]_{\g}\\
&=&[[x_1,x_2, x_3]_{\g},x_4,x_5]_{\g}+[x_3,[x_1,x_2, x_4]_{\g},x_5]_{\g}+[x_3,x_4,[x_1,x_2, x_5]_{\g}]_{\g}.
 \label{eq:jacobi1}
\end{eqnarray}
\end{defi}



Let $\g$ be a vector space. We consider the graded vector space $$C^*(\g,\g)=\oplus_{n\ge 0}C^n(\g,\g)=\oplus_{n\ge 0}\Hom (\underbrace{\wedge^{2} \g\otimes \cdots\otimes \wedge^{2}\g}_{n\geq0}\wedge \g, \g).$$
The degree of elements in $C^n(\g,\g)$ are defined to be $n$. Then the graded vector space $C^*(\g,\g)$ equipped with the   graded commutator bracket
\begin{eqnarray}\label{3-Lie-bracket}
[P,Q]_{\Li}=P{\circ}Q-(-1)^{pq}Q{\circ}P,\quad \forall~ P\in C^{p}(\g,\g),Q\in C^{q}(\g,\g),
\end{eqnarray}
is a graded Lie algebra \cite{NR bracket of n-Lie}, where $P{\circ}Q\in C^{p+q}(\g,\g)$ is defined by

{\footnotesize
\begin{equation*}
\begin{aligned}
&(P{\circ}Q)(\mathfrak{X}_1,\cdots,\mathfrak{X}_{p+q},x)\\
=&\sum_{k=1}^{p}(-1)^{(k-1)q}\sum_{\sigma\in \huaS(k-1,q)}(-1)^\sigma P(\mathfrak{X}_{\sigma(1)},\cdots,\mathfrak{X}_{\sigma(k-1)},
Q(\mathfrak{X}_{\sigma(k)},\cdots,\mathfrak{X}_{\sigma(k+q-1)},x_{k+q})\wedge y_{k+q},\mathfrak{X}_{k+q+1},\cdots,\mathfrak{X}_{p+q},x)\\
&+\sum_{k=1}^{p}(-1)^{(k-1)q}\sum_{\sigma\in \huaS(k-1,q)}(-1)^\sigma P(\mathfrak{X}_{\sigma(1)},\cdots,\mathfrak{X}_{\sigma(k-1)},x_{k+q}\wedge
Q(\mathfrak{X}_{\sigma(k)},\cdots,\mathfrak{X}_{\sigma(k+q-1)},y_{k+q}),\mathfrak{X}_{k+q+1},\cdots,\mathfrak{X}_{p+q},x)\\
&+\sum_{\sigma\in \huaS(p,q)}(-1)^{pq}(-1)^\sigma P(\mathfrak{X}_{\sigma(1)},\cdots,\mathfrak{X}_{\sigma(p)},
Q(\mathfrak{X}_{\sigma(p+1)},\cdots,\mathfrak{X}_{\sigma(p+q-1)},\mathfrak{X}_{\sigma(p+q)},x)),\\
\end{aligned}
\end{equation*}
}
where   $\mathfrak{X}_{i}=x_i\wedge y_i\in \wedge^2 \g$, $i=1,2,\cdots,p+q$ and $x\in\g.$

The following result is well known.
\begin{pro}\label{pro:3LieMC}
Let $\g$ be a vector space. Then $\pi \in C^{1}(\g,\g)=\Hom(\wedge^3\g,\g)$ defines a $3$-Lie algebra structure on $\g$ if and only if $\pi$ is a Maurer-Cartan element of the graded Lie algebra $(C^*(\g,\g),[\cdot,\cdot]_{\Li})$, i.e.
satisfies the Maurer-Cartan equation $[\pi ,\pi]_{\Li}=0$.
Moreover, $(C^*(\g,\g),[\cdot,\cdot]_{\Li},{\rm d}_{\pi})$ is a differential graded Lie algebra, where ${\rm d}_{\pi}$ is defined by
\begin{eqnarray}
{\rm d}_\pi:=[\pi,\cdot]_{\Li}.
\end{eqnarray}
\end{pro}


Next we introduce the notion of a compatible $3$-Lie algebra.
\begin{defi} A {\bf compatible $3$-Lie algebra} is a triple $(\mathfrak{g},[\cdot,\cdot,\cdot],\{\cdot,\cdot,\cdot\})$, where $\g$ is a vector space,
$[\cdot,\cdot,\cdot] $ and $ \{\cdot,\cdot,\cdot\} $ are $3$-Lie algebra structures on $\g$
such that
\begin{eqnarray}\label{eq:cl}
[x_1,x_2,\{x_3,x_4,x_5\}]+\{x_1,x_2,[x_3,x_4,x_5]\}
&=&[\{x_1,x_2,x_3\},x_4,x_5]+\{[x_1,x_2,x_3],x_4,x_5\}\\
\nonumber&&[x_3,\{x_1,x_2,x_4\},x_5]+\{x_3,[x_1,x_2,x_4],x_5\}\\
\nonumber&&[x_3,x_4,\{x_1,x_2,x_5\}]+\{x_3,x_4,[x_1,x_2,x_5]\},
\end{eqnarray}
for all $x_i\in \g, 1\leq i\leq 5.$
\end{defi}

\begin{ex}
Let $(\g,[\cdot,\cdot,\cdot])$ and $(\g,\{\cdot,\cdot,\cdot\})$ be two $4$-dimensional $3$-Lie algebras. The two $3$-Lie brackets are given by $[e_1,e_2,e_3]=e_1$ and $\{e_2,e_3,e_4\}=e_1$ for a given
basis $\{e_1, e_2, e_3, e_4\}$ respectively. Then $(\g,[\cdot,\cdot,\cdot],\{\cdot,\cdot,\cdot\})$ is a compatible $3$-Lie algebra.
\end{ex}

\begin{pro}
A triple $(\mathfrak{g},[\cdot,\cdot,\cdot],\{\cdot,\cdot,\cdot\})$ is a compatible $3$-Lie
algebra if and only if
   $ [\cdot,\cdot,\cdot]$ and
$ \{\cdot,\cdot,\cdot\}$ are $3$-Lie algebra structures on $\g$ such that for all
$k_1,k_2\in\mathbb{K}$, the following trilinear operation
\begin{equation}\label{eq:Courant bracket}
\Courant{x,y,z}=k_1[x,y,z]+k_2\{x,y,z\},\quad \forall x,y,z\in\mathfrak{g}
\end{equation}
defines a $3$-Lie algebra structure on $\mathfrak{g}$.\end{pro}
\begin{proof}
It is straightforward.
\end{proof}

\begin{defi}
  A {\bf derivation} of a compatible $3$-Lie algebra $(\g,[\cdot,\cdot,\cdot],\{\cdot,\cdot,\cdot\})$ is a linear map $D:\g\rightarrow\g$ such that
  \begin{eqnarray*}
f([x,y,z])&=&[f(x),y,z]+[x,f(y),z]+[x,y,f(z)],\\ f(\{x,y,z\})&=&\{f(x),y,z\}+\{x,f(y),z\}+\{x,y,f(z)\}, \quad \forall x,y,z\in \g.
 \end{eqnarray*}
\end{defi}

\begin{defi}
  A {\bf homomorphism} between two compatible $3$-Lie algebras $(\g,[\cdot,\cdot,\cdot]_\g,\{\cdot,\cdot,\cdot\}_\g)$ and $(\h,[\cdot,\cdot,\cdot]_\h,\{\cdot,\cdot,\cdot\}_\h)$ is a linear map $\varphi:\g\rightarrow \mathfrak{\h}$ such that $\varphi$ is both a $3$-Lie algebra homomorphism between $(\g,[\cdot,\cdot,\cdot]_\g)$ and $(\h,[\cdot,\cdot,\cdot]_\h)$ and a $3$-Lie algebra homomorphism between $(\g,\{\cdot,\cdot,\cdot\}_\g)$ and $(\h,\{\cdot,\cdot,\cdot\}_\h)$.
\end{defi}

\begin{defi} \cite{LSB}
  Let $(\huaG,[\cdot,\cdot],{\rm d}_1)$ and $(\huaG,[\cdot,\cdot],{\rm d}_2)$ be two differential graded Lie algebras, where $\huaG=\oplus_{i=0}^\infty \g_i$. We call the quadruple $(\huaG,[\cdot,\cdot],{\rm d}_1,{\rm d}_2)$ a {\bf bidifferential graded Lie algebra} if ${\rm d}_1$ and ${\rm d}_2$
  satisfy
  \begin{equation}
  {\rm d}_1\circ {\rm d}_2+{\rm d}_2\circ {\rm d}_1=0.
  \end{equation}
\end{defi}

\begin{defi} \cite{LSB}
Let $(\huaG,[\cdot,\cdot],{\rm d}_1,{\rm d}_2)$ be a bidifferential graded Lie algebra. A pair $(P_1,P_2)\in \g_1\oplus \g_1$ is called a {\bf Maurer-Cartan element} of the bidifferential graded Lie algebra $(\huaG,[\cdot,\cdot],{\rm d}_1,{\rm d}_2)$ if $P_1$ and $P_2$ are Maurer-Cartan elements of the  differential graded Lie algebras $(\huaG,[\cdot,\cdot],{\rm d}_1)$ and  $(\huaG,[\cdot,\cdot],{\rm d}_2)$ respectively, and
\begin{eqnarray}
{\rm d}_1 P_2+{\rm d}_2 P_1 +[P_1,P_2]=0.
\end{eqnarray}
\end{defi}

Let $(\huaG,[\cdot,\cdot])$ be a graded Lie algebra. It is obvious that $(\huaG,[\cdot,\cdot],{\rm d}_1=0,{\rm d}_2=0)$ is a bidifferential graded Lie algebra. Consider the graded Lie algebra $(C^*(\g,\g),[\cdot,\cdot]_{\Li})$, we obtain the following main result.

\begin{thm}\label{pro:lsymNR}
  Let $\g$ be a vector space, $\pi_1,\pi_2\in\Hom(\wedge^3\g,\g)$. Then $(\g,\pi_1,\pi_2)$ is a compatible $3$-Lie algebra if and only if $(\pi_1,\pi_2)$ is a Maurer-Cartan element of the bidifferential graded Lie algebra $(C^*(\g,\g),[\cdot,\cdot]_{\Li}, {\rm d}_1=0,{\rm d}_2=0)$.
\end{thm}
\begin{proof}
By Proposition \ref{pro:3LieMC},
$\pi_1,\pi_2\in\Hom(\wedge^3\g,\g)$ define two $3$-Lie algebra structures
on $\g$ respectively  if and only if
$$
[\pi_1,\pi_1]_\Li=0,\quad[\pi_2,\pi_2]_\Li=0.
$$
Moreover, we have
\begin{eqnarray*}
&&[\pi_1,\pi_2]_{\Li}(x_1,x_2,x_3,x_4,x_5)\\
&=&\pi_1(\pi_2(x_1,x_2,x_3),x_4,x_5)+\pi_1(x_3,\pi_2(x_1,x_2,x_4),x_5)+\pi_1(x_3,x_4,\pi_2(x_1,x_2,x_5))\\
&&+\pi_2(\pi_1(x_1,x_2,x_3),x_4,x_5)+\pi_2(x_3,\pi_1(x_1,x_2,x_4),x_5)+\pi_2(x_3,x_4,\pi_1(x_1,x_2,x_5))\\
&&-\pi_1(x_1,x_2,\pi_2(x_3,x_4,x_5))-\pi_2(x_1,x_2,\pi_1(x_3,x_4,x_5)).
\end{eqnarray*}
Therefore the compatibility condition (\ref{eq:cl}) is equivalent to $[\pi_1,\pi_2]_{\Li}=0$.
  Thus, $(\g,\pi_1,\pi_2)$ is a compatible $3$-Lie algebra if and only if
 \begin{equation}\label{eq:Maurer-Cartatn 1}
 [\pi_1,\pi_1]_{\Li}=0,\quad [\pi_1,\pi_2]_{\Li}=0,\quad[\pi_2,\pi_2]_{\Li}=0,
 \end{equation}
which means that  $(\pi_1,\pi_2)$ is a Maurer-Cartan element of the bidifferential graded Lie algebra $(C^*(\g,\g),[\cdot,\cdot]_{\Li}, {\rm d}_1=0,{\rm d}_2=0)$.
\end{proof}

Now we are ready to give a new bidifferential graded Lie algebra that controls deformations of a compatible $3$-Lie algebra.

\begin{thm}\label{pro:new differential Lie algebra}
Let $(\g,\pi_1,\pi_2)$ be a compatible $3$-Lie algebra.  Then
$(C^*(\g,\g),[\cdot,\cdot]_\Li,
{\rm d}_{\pi_1},{\rm d}_{\pi_2})$ is a bidifferential graded
Lie algebra, where ${\rm d}_{\pi_1}$ and ${\rm d}_{\pi_2}$
are  respectively  defined by
\begin{eqnarray*}
{\rm d}_{\pi_1}P ={[\pi_1,P]}_{\Li},\quad{\rm d}_{\pi_2}P ={[\pi_2,P]}_{\Li},\quad\forall~P\in C^p(\g,\g).
\end{eqnarray*}

Moreover, for all
$\tilde{\pi}_1,\tilde{\pi}_2\in\Hom(\wedge^3\g,\g)$,
$(\g,\pi_1+\tilde{\pi}_1,\pi_2+\tilde{\pi}_2)$ is a compatible $3$-Lie algebra if and only if the pair
$(\tilde{\pi}_1,\tilde{\pi}_2)$ is a Maurer-Cartan element of the
bidifferential graded Lie algebra $(C^*(\g,\g),[\cdot,\cdot]_\Li,$
${\rm d}_{\pi_1},{\rm d}_{\pi_2})$.

 \end{thm}
\begin{proof}
By Proposition \ref{pro:3LieMC}, $(\g,[\cdot,\cdot]_\Li,{\rm d}_{\pi_1})$ and $(\g,[\cdot,\cdot]_\Li,{\rm d}_{\pi_2})$ are differential graded Lie algebras.
Since $[\pi_1,\pi_2]_{\Li}=0,$  we have
\begin{eqnarray*}
{\rm d}_{\pi_1}{\rm d}_{\pi_2}P+{\rm d}_{\pi_2}{\rm d}_{\pi_1}P=[[\pi_1,\pi_2]_{\Li},P]_{\Li}=0,\quad P\in C^p(\g,\g),
\end{eqnarray*}
which implies that ${\rm d}_{\pi_1}{\rm d}_{\pi_2}+{\rm d}_{\pi_2}{\rm d}_{\pi_1}=0.$
Therefore $(C^*(\g,\g),[\cdot,\cdot]_\Li,
{\rm d}_{\pi_1},{\rm d}_{\pi_2})$ is a bidifferential graded
Lie algebra.

If $(\g,\pi_1+\tilde{\pi}_1,\pi_2+\tilde{\pi}_2)$ is a compatible $3$-Lie algebra, by  Theorem \ref{pro:lsymNR}, $(\pi_1+\tilde{\pi}_1,\pi_2+\tilde{\pi}_2)$ is a Maurer-Cartan element of the bidifferential graded Lie algebra $(C^*(\g,\g),[\cdot,\cdot]_{\Li},{\rm d}_1=0,{\rm d}_2=0)$, then we have
\begin{eqnarray*}
  [\pi_1+\tilde{\pi}_1,\pi_1+\tilde{\pi}_1]_{\Li}&=&0,\\
\nonumber[\pi_2+\tilde{\pi}_2,\pi_2+\tilde{\pi}_2]_{\Li}&=&0,\\
\nonumber[\pi_1+\tilde{\pi}_1,\pi_2+\tilde{\pi}_2]_{\Li}&=&0.
\end{eqnarray*}
Furthermore, since $(\pi_1,\pi_2)$ is a Maurer-Cartan element of the bidifferential graded Lie algebra $(C^*(\g,\g),[\cdot,\cdot]_{\Li},{\rm d}_1=0,{\rm d}_2=0)$, we have
\begin{eqnarray*}
 {\rm d}_{\pi_1}\tilde{\pi}_1+\frac{1}{2}[\tilde{\pi}_1,\tilde{\pi}_1]_{\Li}&=&0,\\
\nonumber {\rm d}_{\pi_2}\tilde{\pi}_2+\frac{1}{2}[\tilde{\pi}_2,\tilde{\pi}_2]_{\Li}&=&0,\\
\nonumber {\rm d}_{\pi_1}\tilde{\pi}_2+{\rm d}_{\pi_2}\tilde{\pi}_1+[\tilde{\pi}_1,\tilde{\pi}_2]_{\Li}&=&0.
\end{eqnarray*}
Thus $(\tilde{\pi}_1,\tilde{\pi}_2)$ is a Maurer-Cartan element of the
bidifferential graded Lie algebra $(C^*(\g,\g),[\cdot,\cdot]_\Li,$
${\rm d}_{\pi_1},{\rm d}_{\pi_2})$. The converse also holds. We omit the details.
\end{proof}

\section{Cohomologies, infinitesimal deformations and Nijenhuis operators}\label{sec:coh}

In this section, we introduce a cohomology theory of a compatible 3-Lie algebra  with coefficients in itself, and show that infinitesimal deformations can be characterized by the second cohomology group. We also introduce the notion of a Nijenhuis operator on a compatible 3-Lie algebra, which gives rise to a trivial 2-order 1-parameter deformation.
\subsection{Cohomologies of compatible $3$-Lie algebras}
  First, we recall the cohomology theory of $3$-Lie algebras.
Let $(\g,[\cdot,\cdot,\cdot])$ be a $3$-Lie algebra with $\pi(x,y,z)=[x,y,z]$. Denote by
$$\huaC_{\Li}^{n}(\g;\g)=
\Hom (\underbrace{\wedge^{2} \g\otimes \cdots\otimes \wedge^{2}\g}_{n-1}\wedge \g,\g),~(n\geq 1),$$ which is the space of $n$-cochains.
Define the coboundary operator $\dM^n_\pi:\huaC_{\Li}^{n}(\g,\g)\rightarrow \huaC_{\Li}^{n+1}(\g,\g)$ by
$$\dM^n_\pi\omega =(-1)^{n-1}[\pi,\omega]_\Li,\quad \forall~\omega\in \huaC_{\Li}^{n}(\g,\g).$$
More precisely, for all $~\mathfrak{X}_{i}=x_{i}\wedge y_{i}\in \wedge^{2}\g,~i=1,2,\cdots,n$~and~$x_{n+1}\in \g$, we have
\begin{eqnarray*}
&&\dM^n_\pi\omega(\mathfrak{X}_1,\cdots,\mathfrak{X}_n,x_{n+1})\\
&=&\sum_{1\leq j<k\leq n}(-1)^{j} \omega\Big(\mathfrak{X}_1,\cdots,\hat{\mathfrak{X}_{j}},\cdots,\mathfrak{X}_{k-1},
\pi(x_j,y_j,x_k)\wedge y_k+x_k\wedge\pi(x_j,y_j,y_k),
\mathfrak{X}_{k+1},\cdots,\mathfrak{X}_{n},x_{n+1}\Big)\\
&&+\sum_{j=1}^{n}(-1)^{j}\omega\Big(\mathfrak{X}_1,\cdots,\hat{\mathfrak{X}_{j}},\cdots,\mathfrak{X}_{n},
\pi(x_j,y_j,x_{n+1})\Big)\\&&
+\sum_{j=1}^{n}(-1)^{j+1}\pi\Big(x_j,y_j,\omega(\mathfrak{X}_1,\cdots,\hat{\mathfrak{X}_{j}},
\cdots,\mathfrak{X}_{n},x_{n+1})\Big)\\&&
+(-1)^{n+1}\Big(\pi(y_n,x_{n+1},\omega(\mathfrak{X}_1,\cdots,\mathfrak{X}_{n-1},x_n))+\pi(x_{n+1},x_n,\omega(\mathfrak{X}_1,\cdots,\mathfrak{X}_{n-1},y_n))\Big).
\end{eqnarray*}
Because of the graded Jacobi
identity, we can deduce that $\dM^{n+1}_\pi\circ\dM^n_\pi=0.$ Thus, $(\oplus_{n=1}^{+\infty}\huaC_{\Li}^{n}(\g;\g),{\dM}_{\pi}^*)$ is a cochain complex, whose cohomology is defined to be the cohomology of 3-Lie algebras. See  \cite{Izquierdo,review} for more details.

Next  we introduce a cohomology
theory of compatible $3$-Lie algebras.
Let $(\mathfrak{g},[\cdot,\cdot,\cdot],\{\cdot,\cdot,\cdot\})$ be a compatible $3$-Lie algebra with $\pi_1(x,y,z)=[x,y,z]$ and $\pi_2(x,y,z)=\{x,y,z\}$. By Theorem \ref{pro:lsymNR},  $(\pi_1,\pi_2)$ is a Maurer-Cartan element of the bidifferential graded Lie algebra $(C^*(\g,\g),[\cdot,\cdot]_{\Li},{\rm d}_1=0,{\rm d}_2=0)$.
Define the space of $n$-cochains $\huaC_{\cC}^n(\g,\g), (n\geq1)$ by
$$\huaC_{\cC}^n(\g,\g)=\underbrace{\huaC_{\Li}^{n}(\g,\g)\oplus \huaC_{\Li}^{n}(\g,\g)\cdots\oplus \huaC_{\Li}^{n}(\g,\g)}_{n~{copies}}.$$


Define $\jetd^1:\huaC_{\cC}^{1}(\g,\g)\longrightarrow \huaC_{\cC}^{2}(\g,\g)$  by
$$\jetd^1 f=([\pi_1, f]_\Li,[\pi_2, f]_\Li),\quad \forall f\in\Hom(\g,\g).$$

Define the coboundary operator $\jetd^n:\huaC_{\cC}^{n}(\g,\g)\longrightarrow \huaC_{\cC}^{n+1}(\g,\g)$ for $n>1$ by
\begin{eqnarray*}
    \jetd^n(\omega_1,\cdots,\omega_{n})=(-1)^{n-1}([\pi_1,\omega_1]_\Li,\cdots,\underbrace{[\pi_2,\omega_{i-1}]_\Li+[\pi_1,\omega_i]_\Li}_i,\cdots,[\pi_2,\omega_{n}]_\Li),
\end{eqnarray*}
where $(\omega_1,\omega_2,\cdots,\omega_n)\in\huaC_{\cC}^n(\g,\g)$ and $2\leq i\leq n$.

The definition of $\jetd^*$ can be presented by the following diagram:
\begin{eqnarray*}
\huaC_{\cC}^{1}(\g,\g)\stackrel{\jetd^1}{\longrightarrow}\huaC_{\cC}^{2}(\g,\g)\stackrel{\jetd^2}{\longrightarrow}\huaC_{\cC}^{3}(\g,\g) \stackrel{\jetd^3}{\longrightarrow}\cdots
\end{eqnarray*}
{\footnotesize{\begin{tikzpicture}[>=stealth,sloped]
    \matrix (tree) [%
      matrix of nodes,
      minimum size=0.2cm,
      column sep=1.6cm,
      row sep=0.2cm,
    ]
    {     & & & &$\huaC_{\Li}^4(\g,\g)$ &\\
          & &  & $\huaC_{\Li}^3(\g,\g)$ & &\\
          & &$\huaC_{\Li}^2(\g,\g)$ &   & $\huaC_{\Li}^4(\g,\g)$ &\\
     & $\huaC_{\Li}^1(\g,\g)$  &  & $\huaC_{\Li}^3(\g,\g)$ & &$\cdots$\\
         & & $\huaC_{\Li}^2(\g,\g)$ &   & $\huaC_{\Li}^4(\g,\g)$ &\\
         & &   & $\huaC_{\Li}^3(\g,\g)$ & &\\
         & & & & $\huaC_{\Li}^4(\g,\g)$ &\\
    };
    \draw[->] (tree-4-2) -- (tree-3-3) node [midway,above] {$[\pi_1,\cdot]_\Li$};
    \draw[->] (tree-4-2) -- (tree-5-3) node [midway,below] {$[\pi_2,\cdot]_\Li$};
    \draw[->] (tree-3-3) -- (tree-2-4) node [midway,above] {$-[\pi_1,\cdot]_\Li$};
    \draw[->] (tree-3-3) -- (tree-4-4) node [midway,above] {$-[\pi_2,\cdot]_\Li$};
    \draw[->] (tree-5-3) -- (tree-4-4) node [midway,above] {$-[\pi_1,\cdot]_\Li$};
    \draw[->] (tree-5-3) -- (tree-6-4) node [midway,below] {$-[\pi_2,\cdot]_\Li$};
    \draw[->] (tree-2-4) -- (tree-1-5) node [midway,above] {$[\pi_1,\cdot]_\Li$};
    \draw[->] (tree-2-4) -- (tree-3-5) node [midway,above] {$[\pi_2,\cdot]_\Li$};
    \draw[->] (tree-4-4) -- (tree-3-5) node [midway,above] {$[\pi_1,\cdot]_\Li$};
    \draw[->] (tree-4-4) -- (tree-5-5) node [midway,above] {$[\pi_2,\cdot]_\Li$};
    \draw[->] (tree-6-4) -- (tree-5-5) node [midway,above] {$[\pi_1,\cdot]_\Li$};
    \draw[->] (tree-6-4) -- (tree-7-5) node [midway,below] {$[\pi_2,\cdot]_\Li$};
  \end{tikzpicture}}}

\begin{thm}\label{thm:cohomology of CLA}
  With the above notations, we have $\jetd^{n+1}\circ\jetd^n=0$ $(n\geq1)$, i.e. $(\huaC_{\cC}^*(\g,\g)=\oplus_{n=1}^\infty \huaC_{\cC}^n(\g,\g),\jetd^*)$ is a cochain complex.
\end{thm}
\begin{proof}
By \eqref{eq:Maurer-Cartatn 1} and the graded Jacobi identity, for $(\omega_1,\omega_2,\cdots,\omega_n)\in\huaC_{\cC}^n(\g,\g),~n\ge 1$, we have
\begin{eqnarray*}
    &&\jetd^{n+1}\jetd^n(\omega_1,\cdots,\omega_{n})\\
    &=&(-1)^{n-1}\jetd^{n+1}([\pi_1,\omega_1]_\Li,\cdots,\underbrace{[\pi_2,\omega_{i-1}]_\Li+[\pi_1,\omega_i]_\Li}_i,\cdots,[\pi_2,\omega_{n}]_\Li)\\
    &=&-([\pi_1,[\pi_1,\omega_1]_\Li]_\Li,[\pi_2,[\pi_1,\omega_1]_\Li]_\Li+[\pi_1,[\pi_2,\omega_1]_\Li]_\Li+[\pi_1,[\pi_1,\omega_2]_\Li]_\Li,\cdots,\\
    &&\underbrace{[\pi_2,[\pi_2,\omega_{i-2}]_\Li]_\Li+[\pi_2,[\pi_1,\omega_{i-1}]_\Li]_\Li+[\pi_1,[\pi_2,\omega_{i-1}]_\Li]_\Li+[\pi_1,[\pi_1,\omega_{i}]_\Li]_\Li}_{3\leq i\leq n-1},\cdots,\\
&&[\pi_2,[\pi_2,\omega_{n-1}]_\Li]_\Li+[\pi_2,[\pi_1,\omega_n]_\Li]_\Li+[\pi_1,[\pi_2,\omega_n]_\Li]_\Li,[\pi_2,[\pi_2,\omega_n]_\Li]_\Li)\\
    &=&-(\half[[\pi_1,\pi_1]_\Li,\omega_1]_\Li,[[\pi_1,\pi_2]_\Li,\omega_1]_\Li+\half [[\pi_1,\pi_1]_\Li,\omega_2]_\Li,\cdots,\\
    &&\underbrace{\half[[\pi_2,\pi_2]_\Li,\omega_{i-2}]_\Li+[[\pi_1,\pi_2]_\Li,\omega_{i-1}]_\Li+\half[[\pi_1,\pi_1]_\Li,\omega_i]_\Li}_{3\leq i\leq n-1},\cdots,\\
    &&\half[[\pi_2,\pi_2]_\Li,\omega_{n-1}]_\Li+[[\pi_1,\pi_2]_\Li,\omega_n]_\Li,\half[[\pi_2,\pi_2]_\Li,\omega_n]_\Li)\\
    &=&(0,0,\cdots,0).
    \end{eqnarray*}
    Thus we have $\jetd^{n+1}\circ\jetd^n=0$.\end{proof}
\begin{defi} Let $(\mathfrak{g},[\cdot,\cdot,\cdot],\{\cdot,\cdot,\cdot\})$ be a compatible $3$-Lie
algebra.
  The cohomology of the cochain complex $(\huaC_{\cC}^*(\g,\g),\jetd^*)$  is  called {\bf the cohomology of   $(\g,[\cdot,\cdot,\cdot],\{\cdot,\cdot,\cdot\})$}.  We denote the $n$-th cohomology group by  $\huaH^n(\g,\g)$.
\end{defi}

  It is clearly that $f\in\Hom(\g,\g)$ is a $1$-cocycle if and only if   $f$ is a derivation  on the compatible $3$-Lie algebra $(\g,[\cdot,\cdot,\cdot],\{\cdot,\cdot,\cdot\})$.

\subsection{Infinitesimal deformations of compatible $3$-Lie algebras}

In this subsection, we study infinitesimal deformations of compatible $3$-Lie algebras using the cohomology theory of compatible $3$-Lie algebras.
\emptycomment{
Let $(\g,[\cdot,\cdot,\cdot])$ be a $3$-Lie algebra over $\mathbb R$ and $\mathbb R[t]$ be the polynomial ring in one variable $t.$
Then $\mathbb R[t]/(t^2)\otimes_{\mathbb R}\g$ is an $\mathbb R[t]/(t^2)$-module, moreover, $\mathbb R[t]/(t^2)\otimes_{\mathbb R}\g$ is a $3$-Lie algebra over $\mathbb R[t]/(t^2)$, where the compatible $3$-Lie algebra structure is defined by
\begin{eqnarray*}
[f_1(t)\otimes x,f_2(t)\otimes y,f_3(z)\otimes y]&=&f_1(t)f_2(t)f_3(t)\otimes_{\mathbb R}[x,y,z],\\
\{f_1(t)\otimes x,f_2(t)\otimes y,f_3(z)\otimes y\}&=&f_1(t)f_2(t)f_3(t)\otimes_{\mathbb R}\{x,y,z\},\quad
f_{i}(t)\in \mathbb R[t]/(t^2),1\leq i\leq 3,x,y,z\in \g.
\end{eqnarray*}

In the sequel, all the vector spaces are finite dimensional vector spaces over $\mathbb R$ and we denote $f(t)\otimes_{\mathbb R} x$ by $f(t)x,$ where $f(t)\in \mathbb R[t]/(t^2).$}

\begin{defi}
  Let $(\mathfrak{g},[\cdot,\cdot,\cdot],\{\cdot,\cdot,\cdot\})$ be a compatible $3$-Lie algebra,
$\omega_1,\omega_2\in \Hom(\wedge^3\g,\g)$.
Define
\begin{equation}
[x,y,z]_t=[x,y,z]+t\omega_1(x,y,z),\;\;\{x,y,z\}_t=\{x,y,z\}+t\omega_2(x,y,z),\;\;\forall
x,y,z\in \mathfrak{g}.
\end{equation}
If for any $t$, $(\mathfrak{g},[\cdot,\cdot,\cdot]_t,\{\cdot,\cdot,\cdot\}_t)$ is still a
compatible $3$-Lie algebra structure modulo $t^2$, then we say that $(\omega_1,\omega_2)$
  generates  an {\bf infinitesimal deformation} of
$(\mathfrak{g},[\cdot,\cdot,\cdot],\{\cdot,\cdot,\cdot\})$.
\end{defi}
It is straightforward to verify that $(\omega_1,\omega_2)$
generates an infinitesimal deformation of a
compatible $3$-Lie algebra
 $(\frkg,[\cdot,\cdot,\cdot],\{\cdot,\cdot,\cdot\})$ if and only if for any $k_1,k_2\in
 \mathbb K$, $k_1\omega_1+k_2\omega_2$ generates an infinitesimal
 deformation of the $3$-Lie algebra  $(\frak g,
 \Courant{\cdot,\cdot,\cdot}=k_1[\cdot,\cdot,\cdot]+k_2\{\cdot,\cdot,\cdot\})$.

We set
$$\pi_1(x,y,z)=[x,y,z],\quad \pi_2(x,y,z)=\{x,y,z\}.$$
By Theorem \ref{pro:lsymNR},  $(\g,[\cdot,\cdot,\cdot]_t,\{\cdot,\cdot,\cdot\}_t)$ is an infinitesimal deformation of $(\g,\pi_1,\pi_2)$  if and only if
$$
{}[{\pi}_1,{\omega}_1]_{\Li}= 0,\quad [{\pi}_1,{\omega}_2]_{\Li}+[{\pi}_2,{\omega}_1]_{\Li}=0,\quad\hfill[{\pi}_2,{\omega}_2]_{\Li}=0,$$
 i.e.
$\jetd^2(\omega_1,\omega_2)=0$, which means that
$(\omega_1,\omega_2)\in \huaC_{\cC}^2(\g,\g)$  is a $2$-cocycle for the
compatible $3$-Lie algebra $(\mathfrak{g},[\cdot,\cdot,\cdot],\{\cdot,\cdot,\cdot\})$.

\begin{defi}
Two infinitesimal deformations $(\g,[\cdot,\cdot,\cdot]_t,\{\cdot,\cdot,\cdot\}_t)$ and
$(\g,[\cdot,\cdot,\cdot]_t',\{\cdot,\cdot,\cdot\}_t')$ of a compatible $3$-Lie algebra
$(\g,[\cdot,\cdot,\cdot],\{\cdot,\cdot,\cdot\})$ generated by $(\omega_1,\omega_2)$ and
$(\omega'_1,\omega'_2) $   respectively  are said to be
{\bf equivalent} if there exists $N\in \gl(\g)$ such that
${\Id}+tN: (\g,[\cdot,\cdot,\cdot]_t,\{\cdot,\cdot,\cdot\}_t)\longrightarrow
(\g,[\cdot,\cdot,\cdot]_t',\{\cdot,\cdot,\cdot\}_t')$ is a compatible $3$-Lie algebra homomorphism modulo $t^2$.
\end{defi}

Two infinitesimal deformations $(\g,[\cdot,\cdot,\cdot]_t,\{\cdot,\cdot,\cdot\}_t)$ and
$(\g,[\cdot,\cdot,\cdot]_t',\{\cdot,\cdot,\cdot\}_t')$ generated by $(\omega_1,\omega_2)$ and
$(\omega'_1,\omega'_2) $    respectively  are equivalent
if and only if
\begin{align}
\omega_1(x,y,z)=&\omega_1'(x,y,z)+[Nx,y,z]+[x,Ny,z]+[x,y,Nz]-N([x,y,z]),\label{2-exact1}\\
\omega_2(x,y,z)=&\omega_2'(x,y,z)+\{Nx,y,z\}+\{x,Ny,z\}+\{x,y,Nz\}-N(\{x,y,z\}).\label{2-exact2}
\end{align}
It is easy to see that $(\ref{2-exact1})$ and $(\ref{2-exact2})$ mean that $(\omega_1- \omega_1',\omega_2-\omega_2')=\jetd^1 N$.

We summarize the above discussion by the following
 theorem.
\begin{thm}\label{thm:deformation}
Let $(\g,[\cdot,\cdot,\cdot]_t,\{\cdot,\cdot,\cdot\}_t)$ be an infinitesimal deformation  of a compatible $3$-Lie algebra $(\g,[\cdot,\cdot,\cdot],\{\cdot,\cdot,\cdot\})$ generated by $(\omega_1,\omega_2)$. Then $(\omega_1,\omega_2)$ is
closed, i.e. $\jetd^2(\omega_1,\omega_2)=0.$

Furthermore,  two infinitesimal deformations
$(\g,[\cdot,\cdot,\cdot]_t,\{\cdot,\cdot,\cdot\}_t)$ and $(\g,[\cdot,\cdot,\cdot]_t',\{\cdot,\cdot,\cdot\}_t')$ of a
compatible $3$-Lie algebra $(\g,[\cdot,\cdot,\cdot],\{\cdot,\cdot,\cdot\})$ generated by
$(\omega_1,\omega_2)$ and $(\omega'_1,\omega'_2) $
respectively  are equivalent if and only if $(\omega_1,\omega_2)$ and
$(\omega_1',\omega'_2)$ are in the same cohomology class in
$\huaH^2(\g,\g)$.
\end{thm}

\subsection{2-order 1-parameter deformations of compatible $3$-Lie algebras}
In this subsection, we study 2-order 1-parameter deformations of compatible $3$-Lie algebras,  and introduce the notion of a Nijenhuis operator on a compatible $3$-Lie algebra, which could generate a trivial deformation.

\emptycomment{
\begin{defi}
Two 2-order 1-parameter deformations $(\g,[\cdot,\cdot,\cdot]_t,\{\cdot,\cdot,\cdot\}_t)$ and
$(\g,[\cdot,\cdot,\cdot]_t',\{\cdot,\cdot,\cdot\}_t')$ of a compatible $3$-Lie algebra
$(\g,[\cdot,\cdot,\cdot],\{\cdot,\cdot,\cdot\})$ generated by $(\omega_1,\tilde{\omega}_1,\omega_2,\tilde{\omega}_2)$ and
$(\omega'_1,\tilde{\omega}'_1,\omega'_2,\tilde{\omega}'_2)$ respectively are said to be
{\bf equivalent} if there exists a linear map $N:\g\rightarrow\g$ such that
${\Id}+tN: (\g,[\cdot,\cdot,\cdot]_t,\{\cdot,\cdot,\cdot\}_t)\longrightarrow
(\g,[\cdot,\cdot,\cdot]_t',\{\cdot,\cdot,\cdot\}_t')$ is a compatible $3$-Lie algebra homomorphism.
\end{defi}}

First we recall Nijenhuis operators on 3-Lie algebras.
Let $(\g,[\cdot,\cdot,\cdot])$ be a $3$-Lie algebra. We also use  $\pi:\wedge^3\g\rightarrow \g$ to indicate the $3$-Lie bracket $[\cdot,\cdot,\cdot]$, i.e. $\pi(x,y,z)=[x,y,z].$

We use $ T_{\pi}N:\g\otimes\g\otimes\g\rightarrow\g$ to denote the Nijenhuis torsion of $N$ defined by
\begin{eqnarray*}
  T_{\pi}N:=\frac{1}{2}[[[\pi,N]_{\Li},N]_{\Li},N]_{\Li}-\frac{1}{2}[[\pi,N^2]_{\Li},N]_{\Li}-[[\pi,N]_{\Li},N^2]_{\Li}+[\pi,N^3]_{\Li}.
\end{eqnarray*}
More precisely,
\begin{eqnarray*}
  T_{\pi}N(x,y,z)&=&3[Nx,Ny,Nz]-3N\Big([Nx,Ny,z]+[x,Ny,Nz]+[Nx,y,Nz]\\
  &&-N[Nx,y,z]-N[x,Ny,z]-N[x,y,Nz]+N^2[x,y,z]\Big).
\end{eqnarray*}

\begin{defi}{\rm(\cite{Liu-Jie-Feng})} Let $(\g,[\cdot,\cdot,\cdot])$ be a $3$-Lie algebra. A linear map $N:\g\longrightarrow
\g$ is called a {\bf Nijenhuis operator}   on
$(\g,[\cdot,\cdot,\cdot])$ if  $T_{\pi}N=0$, i.e.
\begin{eqnarray}\label{3-Lie-Nijenhuis}
[Nx,Ny,Nz]&=&N\Big([Nx,Ny,z]+[x,Ny,Nz]+[Nx,y,Nz]\\
\nonumber&&-N[Nx,y,z]-N[x,Ny,z]-N[x,y,Nz]+N^2[x,y,z]\Big), \quad \forall  x,y,z\in \g.
\end{eqnarray}
\end{defi}
\begin{pro}{\rm(\cite{Liu-Jie-Feng})}\label{Nijenhuis-property}
Let $N:\g\longrightarrow\g$ be a Nijenhuis operator on the $3$-Lie algebra $(\g,[\cdot,\cdot,\cdot])$. Then $(\g,[\cdot,\cdot,\cdot]_{N})$ is also a $3$-Lie algebra which is called the {\bf deformed $3$-Lie algebra}, where the bracket $[\cdot,\cdot,\cdot]_{N}$ is given by
\begin{eqnarray}
\label{Nijenhuis-operator}[x,y,z]_{N}&=&[Nx,Ny,z]+[x,Ny,Nz]+[Nx,y,Nz]\\
\nonumber&&-N[Nx,y,z]-N[x,Ny,z]-N[x,y,Nz]+N^2[x,y,z], \quad \forall x,y,z\in \g.
\end{eqnarray}
Moreover,  $N$ is a $3$-Lie algebra homomorphism from $(\g,[\cdot,\cdot,\cdot]_{N})$ to $(\g,[\cdot,\cdot,\cdot])$.
\end{pro}

Now we give an intrinsic characterization of the deformed bracket $[\cdot,\cdot,\cdot]_{N}$.

\begin{lem}
Denote the deformed bracket $[\cdot,\cdot,\cdot]_{N}$ by $\pi_N$. Then we have
\begin{eqnarray}\label{deformed-3-Lie-pi}
\pi_{N}:=\frac{1}{2}\Big([[\pi,N]_{\Li},N]_{\Li}-[\pi,N^2]_{\Li}\Big).
\end{eqnarray}
\end{lem}
\begin{proof}
  It follows from straightforward computations. We omit details.
\end{proof}

\begin{pro}\label{pro:c3c}
 Let $N$ be a Nijenhuis operator on a $3$-Lie algebra $(\g,[\cdot,\cdot,\cdot])$. Then $(\g,[\cdot,\cdot,\cdot],[\cdot,\cdot,\cdot]_{N})$ is a compatible $3$-Lie algebra if and only if $[\pi,N]_{\Li}$ is a $3$-Lie algebra structure.
\end{pro}
\begin{proof}
 Let $N$ be a Nijenhuis operator on a $3$-Lie algebra $(\g,[\cdot,\cdot,\cdot]).$ By \eqref{deformed-3-Lie-pi}, we have
 \begin{eqnarray*}
[\pi,\pi_{N}]_{\Li}&=&\frac{1}{2}[\pi,[[\pi,N]_{\Li},N]_{\Li}]_{\Li}-\frac{1}{2}[\pi,[\pi,N^2]_{\Li},]_{\Li}\\
&=&\frac{1}{2}([[\pi,[\pi,N]_{\Li}]_{\Li},N]_{\Li}-[[\pi,N]_{\Li},[\pi,N]_{\Li}]_{\Li})\\
&=&-\frac{1}{2}[[\pi,N]_{\Li},[\pi,N]_{\Li}]_{\Li}.
\end{eqnarray*}
Therefore, by Proposition \ref{pro:3LieMC} and Theorem \ref{pro:lsymNR},   $(\g,[\cdot,\cdot,\cdot],[\cdot,\cdot,\cdot]_{N})$ is a compatible $3$-Lie algebra if and only if   $[\pi,N]_{\Li}$ is a $3$-Lie algebra structure.
\end{proof}

\begin{rmk} The above result is different from the case of Lie algebras.
  In the Lie algebra case, a Nijenhuis operator on a Lie algebra $(\g,[\cdot,\cdot])$ also induces a deformed Lie algebra $(\g,[\cdot,\cdot]_N)$, where the Lie bracket $[\cdot,\cdot]_N$ is given by
  $$
  [x,y]_N=[Nx,y]+[x,Ny]-N[x,y],\quad \forall x,y\in\g.
  $$
  Moreover, $(\g,[\cdot,\cdot],[\cdot,\cdot]_N)$ is naturally a compatible Lie algebra. See \cite{Dorfman,LSB} for more details.
\end{rmk}

\emptycomment{

\begin{ex}{\rm
Consider the $3$-dimensional $3$-Lie algebra $(\g,[\cdot,\cdot,\cdot]_{\g})$ given with respect to a basis $\{e_1,e_2,e_3\}$ by
$$[e_1,e_2,e_3]_{\g}=e_1.$$
Thanks to Theorem 3.10 in {\rm\cite{Liu-Jie-Feng}}, any linear transformation $N$ on $\g$ is a Nijenhuis operator.
Suppose
\begin{eqnarray*}
N=(a_{ij})_{3\times3}=\left(\begin{array}{ccc}
 a_{11}&a_{12}&a_{13}\\
 a_{21}&a_{22}&a_{23}\\
 a_{31}&a_{32}&a_{33}
 \end{array}\right).
 \end{eqnarray*}
 We have
 \begin{eqnarray*}
   [\pi,N]_{\Li}(e_i,e_j,e_k)&=&[Ne_i,e_j,e_k]_{\g}+[e_i,Ne_j,e_k]_{\g}+[e_i,e_j,Ne_k]_{\g}-N[e_i,e_j,e_k]_{\g},
   \quad 1\leq i,j,k\leq3.
    \end{eqnarray*}
For convenience, we set $\langle e_i,e_j,e_k\rangle_{N}=[\pi,N]_{\Li}(e_i,e_j,e_k).$
Then the multiplication table of $\langle \cdot,\cdot,\cdot\rangle_{N}$ is given by
 \begin{eqnarray*}
\langle e_i,e_j,e_k\rangle_{N}=(-1)^{\tau(ijk)}(a_{22}e_1+a_{33}e_1-a_{21}e_2+a_{31}e_3),\quad i\neq j\neq k,
 \end{eqnarray*}
 and all the other brackets are zero, where the terms $\tau(ijk)$ is the inverse table permutation of $i,j,k.$ Then we can deduce that the $3$-ary  multiplication $\langle \cdot,\cdot,\cdot\rangle_{N}$ is a $3$-Lie algebra strcuture.

By Proposition \ref{pro:c3c}, we get $(\g,[\cdot,\cdot,\cdot]_{\g},[\cdot,\cdot,\cdot]_{N})$ is a compatible $3$-Lie algebra,
where the deformed $3$-Lie algebra structure $[\cdot,\cdot,\cdot]_{N}$ is defined by
$$[e_1,e_2,e_3]_{N}=(a_{22}a_{33}-a_{32}a_{23})e_1+(a_{31}a_{23}-a_{33}a_{21})e_2+(a_{21}a_{32}-a_{22}a_{31})e_3.$$

}
\end{ex}
}

In the sequel, we study another kind  of deformations, which will leads to the concept of a Nijenhuis operator on a compatible 3-Lie algebra.

\begin{defi}
  Let $(\mathfrak{g},[\cdot,\cdot,\cdot],\{\cdot,\cdot,\cdot\})$ be a compatible $3$-Lie algebra. Let $\omega_{i},\tilde{\omega}_{i}\in \Hom(\wedge^3\g,\g),i=1,2.$
Define
\begin{equation}
[x,y,z]_t=[x,y,z]+t\omega_1(x,y,z)+t^2\tilde{\omega}_1(x,y,z),\;\;\{x,y,z\}_t=\{x,y,z\}+t\omega_2(x,y,z)+t^2\tilde{\omega}_2(x,y,z),
\end{equation}
for all $x,y,z\in \g$.
If for any $t$, $(\mathfrak{g},[\cdot,\cdot,\cdot]_t,\{\cdot,\cdot,\cdot\}_t)$ is still a
compatible $3$-Lie algebra, then we say
that $(\omega_1,\tilde{\omega}_1,\omega_2,\tilde{\omega}_2)$ generate a  {\bf $2$-order $1$-parameter deformation} of
$(\mathfrak{g},[\cdot,\cdot,\cdot],\{\cdot,\cdot,\cdot\})$.
\end{defi}
\emptycomment{
\begin{pro}\label{conds}
With the above notations, $\omega_1,\tilde{\omega}_1,\omega_2,\tilde{\omega}_2$ generate an $2$-order $1$-parameter deformation of the compatible $3$-Lie algebra $(\mathfrak{g},[\cdot,\cdot,\cdot],\{\cdot,\cdot,\cdot\})$  if and only if the following conditions are satisfied:
  \begin{eqnarray}
   \delta\omega_1&=&0;\label{cond1}\\
   \delta\omega_l+\frac{1}{2}\sum_{i=1}^{l-1}[\omega_i,\omega_{l-i}]&=&0,\quad 2\leq l\leq n-1;\label{cond2}\\
   \frac{1}{2}\sum_{i=l-n+1}^{n-1}[\omega_i,\omega_{l-i}]&=&0,\quad n\leq l\leq 2n-2.\label{cond3}
   \end{eqnarray}
Here $[\omega_i,\omega_j]$ is given by
\begin{eqnarray}\label{N-R bracket}
[\omega_i,\omega_j](X,Y,z)&=&\omega_i(X,\omega_j(Y,z))-\omega_i(Y,\omega_j(X,z))+\omega_j(X,\omega_i(Y,z))-\omega_j(Y,\omega_i(X,z))\nonumber\\
&&-\omega_i(\omega_j(X,\cdot)\circ Y,z)-\omega_j(\omega_i(X,\cdot)\circ Y,z),\quad\forall X,Y\in\wedge^{n-1}\g,~z\in\g,
\end{eqnarray}
where $\omega_j(X,\cdot)\circ Y\in\wedge^{n-1}\g$ is given by
\begin{eqnarray*}
\omega_j(X,\cdot)\circ Y=\sum_{k=1}^{n-1}y_1\wedge\cdots\wedge\omega_j(X,y_k)\wedge\cdots\wedge y_{n-1},\quad \forall~ Y=(y_1,\cdots, y_{n-1}).
\end{eqnarray*}
\end{pro}}

We set
$\pi_1(\cdot,\cdot,\cdot)=[\cdot,\cdot,\cdot]$ and $ \pi_2(\cdot,\cdot,\cdot)=\{\cdot,\cdot,\cdot\}.$
By Theorem \ref{pro:lsymNR},  $(\g,[\cdot,\cdot,\cdot]_t,\{\cdot,\cdot,\cdot\}_t)$ is a 2-order 1-parameter deformation of $(\g,\pi_1,\pi_2)$  if and only if
\emptycomment{
\begin{equation}\label{eq:2-closed omega bracket-1}
 \begin{array}{rclrclrcl}
&[{\pi}_1,{\omega}_1]_{\Li}&=&0,\\
&[{\pi}_1,{\omega}_2]_{\Li}+[{\pi}_2,{\omega}_1]_{\Li}&=&0,\\
&[{\pi}_2,{\omega}_2]_{\Li}&=&0,\\
 \end{array}
\end{equation}}

\begin{eqnarray}\label{eq:2-closed omega bracket-1}
\left\{\begin{aligned}
{}[{\pi}_1,{\omega}_1]_{\Li}&=&0,\\
{}[{\pi}_1,{\omega}_2]_{\Li}+[{\pi}_2,{\omega}_1]_{\Li}&=&0,\\
{}[{\pi}_2,{\omega}_2]_{\Li}&=&0;
\end{aligned}\right.
\end{eqnarray}

\begin{eqnarray}\label{eq:2-closed omega bracket-2}
\left\{\begin{aligned}
{}[{\pi_1},\tilde{\omega}_1]_{\Li}+\frac{1}{2}[\omega_1,\omega_1]_{\Li}&=&0,\\
{}[\tilde{\omega}_1,{\pi}_2]_{\Li}+[{\omega}_1,{\omega}_2]_{\Li}+[\pi_1,\tilde{\omega}_2]_{\Li}&=&0,\\
{}[\pi_2,\tilde{\omega}_2]_{\Li}+\frac{1}{2}[\omega_2,\omega_2]_{\Li}&=&0;
\end{aligned}\right.
\end{eqnarray}

\begin{eqnarray}\label{eq:2-closed omega bracket-3}
\left\{\begin{aligned}
{}{[{\omega}_1,\tilde{\omega}_1]}_{\Li}&=&0,\\
{}{[\tilde{\omega}_1,{\omega}_2]}_{\Li}+{[\tilde{\omega}_2,{\omega}_1]}_{\Li}&=&0,\\
{}{[{\omega}_2,\tilde{\omega}_2]}_{\Li}&=&0;
\end{aligned}\right.
\end{eqnarray}

\begin{eqnarray}\label{eq:2-closed omega bracket-4}
\left\{\begin{aligned}
{}{[\tilde{\omega}_1,\tilde{\omega}_1]}_{\Li}&=&0,\\
{}{[\tilde{\omega}_1,\tilde{\omega}_2]}_{\Li}&=&0,\\
{}{[\tilde{\omega}_2,\tilde{\omega}_2]}_{\Li}&=&0.
\end{aligned}\right.
\end{eqnarray}

Note that \eqref{eq:2-closed omega bracket-1} means that
$(\omega_1,\omega_2)\in \huaC_{\cC}^2(\g,\g)$  is a $2$-cocycle for the
compatible $3$-Lie algebra $(\mathfrak{g},[\cdot,\cdot,\cdot],\{\cdot,\cdot,\cdot\})$, i.e.
$\jetd^2(\omega_1,\omega_2)=0$,
\eqref{eq:2-closed omega bracket-4} means that
$(\g,\tilde{\omega}_1,\tilde{\omega}_2)$ is a compatible $3$-Lie algebra, and \eqref{eq:2-closed omega bracket-3}
means that $(\omega_1,\omega_2)\in \huaC_{\cC}^2(\g,\g)$  is a $2$-cocycle for the
compatible $3$-Lie algebra $(\g,\tilde{\omega}_1,\tilde{\omega}_2)$.

\begin{defi}
A $2$-order $1$-parameter deformation $(\g,[\cdot,\cdot,\cdot]_t,\{\cdot,\cdot,\cdot\}_t)$ of a compatible $3$-Lie algebra $(\g,[\cdot,\cdot,\cdot],\{\cdot,\cdot,\cdot\})$ generated by $(\omega_1,\tilde{\omega}_1,\omega_2,\tilde{\omega}_2)$ is said to be {\bf trivial} if there exists a linear map $N:\g\rightarrow\g$ such that ${\Id}+tN:(\g,[\cdot,\cdot,\cdot]_t,\{\cdot,\cdot,\cdot\}_t)\longrightarrow (\g,[\cdot,\cdot,\cdot],\{\cdot,\cdot,\cdot\})$ is a compatible $3$-Lie algebra homomorphism.
\end{defi}

By straightforward computations, $(\mathfrak{g},[\cdot,\cdot,\cdot]_t,\{\cdot,\cdot,\cdot\}_t)$ is a trivial 2-order 1-parameter deformation if and only if
\begin{eqnarray}
\label{eq:nijenhuis-1}\omega_1(x,y,z)&=&[N(x),y,z]+[x,N(y),z]+[x,y,N(z)]-N([x,y,z]),\label{C-Nijenhuis-1}\\
\label{eq:nijenhuis-2}\tilde{\omega}_1(x,y,z)+N\omega_1(x,y,z)&=&[N(x),N(y),z]+[N(x),y,N(z)]+[x,N(y),N(z)],\label{C-Nijenhuis-2}\\
\label{eq:nijenhuis-3}N\tilde{\omega}_1(x,y,z)&=&[Nx,Ny,Nz],\label{C-Nijenhuis-3}\\
\label{eq:nijenhuis-4}\omega_2(x,y,z)&=&\{N(x),y,z\}+\{x,N(y),z\}+\{x,y,N(z)\}-N(\{x,y,z\}),\label{C-Nijenhuis-4}\\
\label{eq:nijenhuis-5}\tilde{\omega}_2(x,y,z)+N\omega_2(x,y,z)&=&\{N(x),N(y),z\}+\{N(x),y,N(z)\}+\{x,N(y),N(z)\},\label{C-Nijenhuis-5}\\
\label{eq:nijenhuis-6}N\tilde{\omega}_2(x,y,z)&=&\{Nx,Ny,Nz\}\label{C-Nijenhuis-6}.
\end{eqnarray}

By $(\ref{C-Nijenhuis-1})$-$(\ref{C-Nijenhuis-3})$, $N$ is a Nijenhuis operator on the $3$-Lie algebra $(\g,[\cdot,\cdot,\cdot])$.  By $(\ref{C-Nijenhuis-4})$-$(\ref{C-Nijenhuis-6})$, $N$ is a Nijenhuis operator on the $3$-Lie algebra $(\g,\{\cdot,\cdot,\cdot\})$.
This leads to the concept of
a Nijenhuis operator  on  a compatible 3-Lie algebra.

\begin{defi}
Let $(\mathfrak{g},[\cdot,\cdot,\cdot],\{\cdot,\cdot,\cdot\})$ be a
compatible $3$-Lie algebra. A linear map $N: \g\rightarrow
\mathfrak{g}$  is called a {\bf Nijenhuis operator} on
$(\mathfrak{g},[\cdot,\cdot,\cdot],\{\cdot,\cdot,\cdot\})$ if $N$ is both a Nijenhuis operator on the $3$-Lie algebra
$(\mathfrak{g},[\cdot,\cdot,\cdot])$ and a Nijenhuis operator on the $3$-Lie algebra
$(\mathfrak{g},\{\cdot,\cdot,\cdot\})$.
\end{defi}

\begin{pro}\label{pro:Nijenhuis torsion}
 Let $(\g,[\cdot,\cdot,\cdot],\{\cdot,\cdot,\cdot\})$ be a compatible $3$-Lie algebra and $N:\g\rightarrow \g$ be a linear map.
 Then $N$ is a Nijenhuis operator on the compatible $3$-Lie algebra $(\g,[\cdot,\cdot,\cdot],\{\cdot,\cdot,\cdot\})$ if and only if for all $k_1,k_2\in\K$, $N$ is a Nijenhuis operator on the $3$-Lie algebra  $(\frak g,
 \Courant{\cdot,\cdot,\cdot}=k_1[\cdot,\cdot,\cdot]+k_2\{\cdot,\cdot,\cdot\})$.
\end{pro}
\begin{proof}
First we have   \begin{eqnarray*}
    \Courant{x,y,z}_N&=&\Courant{Nx,Ny,z}+\Courant{x,Ny,Nz}+\Courant{Nx,y,Nz}\\
\nonumber&&-N\Courant{Nx,y,z}-N\Courant{x,Ny,z}-N\Courant{x,y,Nz}+N^2\Courant{x,y,z}\\
&=&k_1[x,y,z]_N+k_2\{x,y,z\}_N,
  \end{eqnarray*}
  which implies that
  $$
  N\Courant{x,y,z}_N-\Courant{Nx,Ny,Nz}_N=k_1(N[x,y,z]_N-[Nx,Ny,Nz])+k_2(N\{x,y,z\}_N-\{Nx,Ny,Nz\}).
  $$
  Therefore, for all $k_1,k_2\in\K$, $N$ is a Nijenhuis operator on the $3$-Lie algebra  $(\frak g,
 \Courant{\cdot,\cdot,\cdot})$ if and only if $N$ is a Nijenhuis operator on the compatible $3$-Lie algebra $(\g,[\cdot,\cdot,\cdot],\{\cdot,\cdot,\cdot\})$.
\end{proof}

\begin{pro}\label{pro:Nijenhuis operator property}
 Let  $N $ be a Nijenhuis operator on a compatible $3$-Lie algebra $(\mathfrak{g},[\cdot,\cdot,\cdot],\{\cdot,\cdot,\cdot\})$. Then $(\g,[\cdot,\cdot,\cdot]_N,\{\cdot,\cdot,\cdot\}_N)$ is also a compatible $3$-Lie algebra and $N$ is a compatible $3$-Lie algebra
 homomorphism from
 $(\g,[\cdot,\cdot,\cdot]_N,\{\cdot,\cdot,\cdot\}_N)$  to $(\g,[\cdot,\cdot,\cdot],\{\cdot,\cdot,\cdot\})$.
\end{pro}
\begin{proof}
  Let  $N:\g\rightarrow \g$ be a Nijenhuis operator on the compatible $3$-Lie algebra $(\mathfrak{g},[\cdot,\cdot,\cdot],\{\cdot,\cdot,\cdot\})$. By Proposition \ref{pro:Nijenhuis torsion}, $N$ is a Nijenhuis operator on the $3$-Lie algebra $(\g,\Courant{\cdot,\cdot,\cdot})$, where the bracket $\Courant{\cdot,\cdot,\cdot}$ is given by \eqref{eq:Courant bracket}. By Proposition \ref{Nijenhuis-property}, we have
  \begin{eqnarray*}
    \Courant{\cdot,\cdot,\cdot}_N
&=&k_1[\cdot,\cdot,\cdot]_N+k_2\{\cdot,\cdot,\cdot\}_N,
  \end{eqnarray*}
is a 3-Lie algebra for all  $k_1,k_2\in\K$. Therefore,
$(\g,[\cdot,\cdot,\cdot]_N,\{\cdot,\cdot,\cdot\}_N)$ is a compatible $3$-Lie algebra and $N$ is a
  homomorphism from $(\g,[\cdot,\cdot,\cdot]_N,\{\cdot,\cdot,\cdot\}_N)$  to $(\g,[\cdot,\cdot,\cdot],\{\cdot,\cdot,\cdot\})$.
\end{proof}

We have seen that a trivial $2$-order $1$-parameter  deformation of a compatible $3$-Lie algebra gives rise to a Nijenhuis operator. The following theorem shows that the converse is also true.
\begin{thm}
Let $N$ be a Nijenhuis operator on a compatible $3$-Lie algebra $(\g,[\cdot,\cdot,\cdot],\{\cdot,\cdot,\cdot\})$. Then a  $2$-order $1$-parameter deformation can be obtained by putting
\begin{eqnarray}
{}\omega_1(x,y,z)&=&[Nx,y,z]+[x,Ny,z]+[x,y,Nz]-N[x,y,z],\\
\tilde{\omega}_1(x,y,z)&=&[Nx,Ny,z]+[Nx,y,Nz]+[x,Ny,Nz]\\
\nonumber&&-N([Nx,y,z]+[x,Ny,z]+[x,y,Nz]-N[x,y,z]),\\
{}\omega_2(x,y,z)&=&\{Nx,y,z\}+\{x,Ny,z\}+\{x,y,Nz\}-N\{x,y,z\},\\
\tilde{\omega}_2(x,y,z)&=&\{Nx,Ny,z\}+\{Nx,y,Nz\}+\{x,Ny,Nz\}\\
\nonumber&&-N(\{Nx,y,z\}+\{x,Ny,z\}+\{x,y,Nz\}-N\{x,y,z\}),
\end{eqnarray}
for all $x,y,z\in \g.$ Moreover, this $2$-order $1$-parameter deformation is trivial.
\end{thm}
\begin{proof}
Since $N$ is a Nijenhuis operator on  $(\g,[\cdot,\cdot,\cdot],\{\cdot,\cdot,\cdot\})$,
$\tilde{\omega}_1(x,y,z)=[x,y,z]_{N}$ and $\tilde{\omega}_2(x,y,z)=\{x,y,z\}_{N}$, by Proposition \ref{pro:Nijenhuis operator property},
$(\g,\tilde{\omega}_1,\tilde{\omega}_2)$ is also a compatible $3$-Lie algebra. By direct calculations,   we can deduce that
 \eqref{eq:2-closed omega bracket-1}-\eqref{eq:2-closed omega bracket-3}
hold. Then $(\g,[\cdot,\cdot,\cdot]_t,\{\cdot,\cdot,\cdot\}_t)$ is a
$2$-order $1$-parameter deformation of  $(\g,[\cdot,\cdot,\cdot],\{\cdot,\cdot,\cdot\}).$

It is straightforward to verify that \eqref{eq:nijenhuis-1}-\eqref{eq:nijenhuis-6} are satisfied. Thus this $2$-order $1$-parameter deformation is
trivial.
\end{proof}

\section{Abelian extensions of compatible 3-Lie algebras}\label{sec:cohomology II}

In this section, we introduce   cohomologies of a compatible $3$-Lie
algebra with coefficients in arbitrary representation, and use the second cohomology group
to classify abelian extensions of a compatible $3$-Lie algebra.

\subsection{Cohomologies   with coefficients in arbitrary representation}\label{sec:cohomology-gc}
  First, we recall some basic results involving representations and cohomologies of $3$-Lie algebras.

\begin{defi}{\rm (\cite{KA})}
A {\bf representation} of a $3$-Lie algebra $(\g,[\cdot,\cdot,\cdot]_{\g})$ on a vector space $V$ is a linear
map: $\rho:\wedge^{2}\g\rightarrow \gl(V)$ such that for all $x_{1}, x_{2}, x_{3}, x_{4}\in \g,$ the following equalities hold:
\begin{eqnarray}
~\label{representation-1}\rho(x_{1},x_{2})\rho(x_{3},x_{4})&=&\rho([x_{1},x_{2},x_{3}]_{\g},x_{4})+
\rho(x_{3},[x_{1},x_{2},x_{4}]_{\g})+\rho(x_{3},x_{4})\rho(x_{1},x_{2}),\\
~\label{representation-2}\rho(x_{1},[x_{2},x_{3},x_{4}]_{\g})&=&\rho(x_{3},x_{4})\rho(x_{1},x_{2})-\rho(x_{2},x_{4})\rho(x_{1},x_{3})
+\rho(x_{2},x_{3})\rho(x_{1},x_{4}).
\end{eqnarray}
\end{defi}

Let $\g_1$ and $\g_2$ be vector spaces. Denote by $\g^{l,k}$ the subspace of $\wedge^{2} (\g_1\oplus\g_2)\otimes$$\overset{(n)}{\cdots}$$ \otimes \wedge^{2} (\g_1\oplus\g_2)\wedge (\g_1\oplus\g_2)$
which contains the number of $\g_1$ (resp. $\g_2$) is $l$ (resp. $k$).
Then the vector space $\wedge^{2} (\g_1\oplus\g_2)\otimes$$\overset{(n)}{\cdots}$$ \otimes \wedge^{2} (\g_1\oplus\g_2)\wedge (\g_1\oplus\g_2)$ is isomorphic to the direct sum of $\g^{l,k},~l+k=2n+1$. For instance,
$$
\wedge^3(\g_1\oplus\g_2)=\g^{3,0}\oplus\g^{2,1}\oplus\g^{1,2}\oplus\g^{0,3}.
$$
An element $f\in\Hom(\g^{l,k},\g_1)$ (resp. $f\in\Hom(\g^{l,k},\g_2)$) naturally gives an element $\hat{f}\in C^n(\g_1\oplus\g_2,\g_1\oplus\g_2)$, which is called its {\bf lift}.
For example, the lifts of linear maps $\pi:\wedge^3\g_1\lon\g_1,~\rho:\wedge^2\g_1\otimes\g_2\lon\g_2$ are defined by
\begin{eqnarray}
\label{semidirect-1}\hat{\pi}\big((x,u),(y,v),(z,w)\big)&=&(\pi(x,y,z),0),\\
\label{semidirect-2}\hat{\rho}\big((x,u),(y,v),(z,w)\big)&=&(0,\rho(x,y)w+\rho(y,z)u+\rho(z,x)v),
\end{eqnarray}
respectively.
It is straightforward to see that we have the following natural isomorphism
\begin{eqnarray}\label{decomposition}
C^n(\g_1\oplus\g_2,\g_1\oplus\g_2)\cong\sum_{l+k=2n+1}\Hom(\g^{l,k},\g_1)\oplus\sum_{l+k=2n+1}\Hom(\g^{l,k},\g_2).
\end{eqnarray}

\begin{defi}\label{Bidegree}\cite{HouST}
A linear map $f\in \Hom(\underbrace{\wedge^{2} (\g_1\oplus\g_2)\otimes{\cdots} \otimes \wedge^{2} (\g_1\oplus\g_2)}_n\wedge (\g_1\oplus\g_2),\g_1\oplus\g_2)$ has a {\bf bidegree} $l|k$, which is denoted by $||f||=l|k$,   if $f$ satisfies the following four conditions:
\begin{itemize}
\item[\rm(i)] $l+k=2n;$
\item[\rm(ii)] If $X$ is an element in $\g^{l+1,k}$, then $f(X)\in\g_1;$
\item[\rm(iii)] If $X$ is an element in $\g^{l,k+1}$, then $f(X)\in\g_2;$
\item[\rm(iv)] All the other case, $f(X)=0.$
\end{itemize}
\end{defi}
A linear map $f$ is said to be homogeneous  if $f$ has a bidegree.
Linear maps $\hat{\pi},~\hat{\rho}\in C^1(\g_1\oplus\g_2,\g_1\oplus\g_2)$ given in \eqref{semidirect-1} and \eqref{semidirect-2} have the bidegree  $||\hat{\pi}||=||\hat{\rho}||=2|0$. Thus, the sum $\hat{\pi}+\hat{\rho}$
is a homogeneous linear map of the bidegree $2|0$, which is also a multiplication of the semidirect product type,
$$
(\hat{\pi}+\hat{\rho})\big((x,u),(y,v),(z,w)\big)=(\pi(x,y,z),\rho(x,y)w+\rho(y,z)u+\rho(z,x)v).
$$

\begin{lem}\label{important-lemma-2}\cite{HouST}
If $||f||=l_f|k_f$ and $||g||=l_g|k_g$, then $[f,g]_{\Li}$ has the bidegree $(l_f+l_g)|(k_f+k_g).$
\end{lem}

The following result is well known. See the survey \cite{GLST} for more details.

\begin{pro}\label{3-Lie-graded}
Let $(\g,\pi)$ be a $3$-Lie algebra and $(V;\rho)$ a representation of $\g$. Then we have
\begin{eqnarray}
[\hat{\pi}+\hat{\rho},\hat{\pi}+\hat{\rho}]_{\Li}=0.
\end{eqnarray}
\end{pro}
Let $(\g,\pi)$ be a $3$-Lie algebra and $(V;\rho)$  a representation
of $\g$.
Denote by
$$\huaC_{\Li}^{n}(\g;V)=
\Hom (\underbrace{\wedge^{2} \g\otimes \cdots\otimes \wedge^{2}\g}_{n-1}\wedge \g,V),~(n\geq 1),$$ which is the space of $n$-cochains.
Define the coboundary operator $\dM^{n}_{\pi+\rho}:\huaC_{\Li}^{n}(\g,V)\rightarrow \huaC_{\Li}^{n+1}(\g,V)$ by
\begin{equation}\label{eq:CE-operator}
  \dM^n_{\pi+\rho} f:=(-1)^{n-1}[\hat{\pi}+\hat{\rho},\hat{f}]_{\Li},\quad \forall~f\in \huaC_{\Li}^{n}(\g,V).
\end{equation}
In fact, since $\hat{\pi}+\hat{\rho}\in \huaC_{\Li}^2(\g\oplus V,\g\oplus V)$ where
$||\hat{\pi}+\hat{\rho}||=2|0,$
 and $\hat{f}\in \huaC_{\Li}^{n}(\g\oplus V,\g\oplus V)$ where $||\hat{f}||=(2n-1)|-1,$ by Lemma \ref{important-lemma-2}, we have $\dM^n_{\pi+\rho} f\in \huaC_{\Li}^{n+1}(\g,V) $. By Proposition \ref{3-Lie-graded} and the graded Jacobi identity, we have $\dM^{n+1}_{\pi+\rho}\circ \dM^n_{\pi+\rho}=0$. Thus we obtain a well-defined cochain complex $(\huaC_{\Li}^*(\g,V),\dM^*_{\pi+\rho})$.
More precisely, for all $f\in \huaC_{\Li}^{n}(\g,V)$ and
 for all $~\mathfrak{X}_{i}=x_{i}\wedge y_{i}\in \wedge^{2}\g,~i=1,2,\cdots,n,~x_{n+1}\in \g$,   we have
  \begin{eqnarray*}
&&\dM^n_{\pi+\rho} f(\mathfrak{X}_1,\cdots,\mathfrak{X}_n,x_{n+1})\\
&=&\sum_{1\leq j<k\leq n}(-1)^{j} f(\mathfrak{X}_1,\cdots,\hat{\mathfrak{X}}_{j},\cdots,\mathfrak{X}_{k-1},
[x_j,y_j,x_k]_\g\wedge y_k+x_k\wedge[x_j,y_j,y_k]_\g,
\mathfrak{X}_{k+1},\cdots,\mathfrak{X}_{n},x_{n+1})\\&&
+\sum_{j=1}^{n}(-1)^{j}f(\mathfrak{X}_1,\cdots,\hat{\mathfrak{X}}_{j},\cdots,\mathfrak{X}_{n},
[x_j,y_j,x_{n+1}]_{\g})\\
&&+\sum_{j=1}^{n}(-1)^{j+1}\rho(x_j,y_j)f(\mathfrak{X}_1,\cdots,\hat{\mathfrak{X}}_{j},
\cdots,\mathfrak{X}_{n},x_{n+1})\\&&
+(-1)^{n+1}\big(\rho(y_n,x_{n+1})f(\mathfrak{X}_1,\cdots,\mathfrak{X}_{n-1},x_n)+\rho(x_{n+1},x_n)f(\mathfrak{X}_1,\cdots,\mathfrak{X}_{n-1},y_n)\big).
\end{eqnarray*}
 See  \cite{Casas,Takhtajan1} for more details.

\begin{defi}
 A {\bf representation} of a compatible $3$-Lie algebra
$(\mathfrak{g},[\cdot,\cdot,\cdot],\{\cdot,\cdot,\cdot\})$ on a
vector space $V$  consists of a pair of linear maps
 $\rho, \mu: \wedge^2\mathfrak{g}\rightarrow \mathfrak{gl}(V)$ such that $\rho$ is a representation of the $3$-Lie algebra $(\mathfrak{g},[\cdot,\cdot,\cdot])$ on $V$, $\mu$ is a representation of the $3$-Lie algebra $(\mathfrak{g}, \{\cdot,\cdot,\cdot\})$ on $V$, and for all $x_{1}, x_{2}, x_{3}, x_{4}\in \g,$ the following equalities hold:
 \begin{eqnarray}
   \nonumber&&\rho(\{x_1,x_2,x_3\},x_4)+\mu([x_1,x_2,x_3],x_4)+\rho(x_3,\{x_1,x_2,x_4\})+\mu(x_3,[x_1,x_2,x_4])\\
\label{eq:COM-Representation-1}&=& [\rho(x_1,x_2),\mu(x_3,x_4)]-[\rho(x_3,x_4),\mu(x_1,x_2)],\\
 \nonumber &&\rho(\{x_1,x_2,x_3\},x_4)+\mu([x_1,x_2,x_3],x_4)-\rho(x_3,x_1)\mu(x_2,x_4)-\mu(x_3,x_1)\rho(x_2,x_4)\\
\label{eq:COM-Representation-2}&=&\rho(x_1,x_2)\mu(x_3,x_4)+\mu(x_1,x_2)\rho(x_3,x_4)+\rho(x_2,x_3)\mu(x_1,x_4)+\mu(x_2,x_3)\rho(x_1,x_4).
 \end{eqnarray}
\end{defi}

It is straightforward to obtain

 \begin{lem}\label{semi-com-3-Lie}
Let $(\g,[\cdot,\cdot,\cdot],\{\cdot,\cdot,\cdot\})$ be a compatible $3$-Lie algebra, $V$ a vector space and  $\rho, \mu: \wedge^2\mathfrak{g}\rightarrow \mathfrak{gl}(V)$  be a pair of linear maps. Then $(V;\rho,\mu)$ is a representation of $\g$ if and only if there is a compatible $3$-Lie algebra structure
on the direct sum $\g\oplus V$ of vector spaces, defined by
\begin{eqnarray}
{}\label{semi-direct-com-3-Lie-1}[x_1+v_1,x_2+v_2,x_3+v_3]_{\g\oplus V}&=&[x_1,x_2,x_3]+\rho(x_1,x_2)v_3+\rho(x_2,x_3)v_1+\rho(x_3,x_1)v_2,\\
{}\label{semi-direct-com-3-Lie-2}\{x_1+v_1,x_2+v_2,x_3+v_3\}_{\g\oplus V}&=&\{x_1,x_2,x_3\}+\mu(x_1,x_2)v_3+\mu(x_2,x_3)v_1+\mu(x_3,x_1)v_2,
 \end{eqnarray}
 for all $x_1,x_2,x_3\in\g,~v_1,v_2,v_3\in V.$ The above compatible $3$-Lie algebra $(\g\oplus V, [\cdot,\cdot,\cdot]_{\g\oplus V},\{\cdot,\cdot,\cdot\}_{\g\oplus V})$ is called the {\bf semi-direct product compatible $3$-Lie algebra} and denoted by $\g\ltimes_{\rho,\mu}V.$
 \end{lem}
 \begin{proof}
 Let $(V;\rho,\mu)$ be a representation of the compatible $3$-Lie algebra $(\g,[\cdot,\cdot,\cdot],\{\cdot,\cdot,\cdot\}).$ It is obviously that $(\g\oplus V,[\cdot,\cdot,\cdot]_{\g\oplus V})$ and $(\g\oplus V,\{\cdot,\cdot,\cdot\}_{\g\oplus V})$ are $3$-Lie algebras respectively.

For all $x_i\in\g$ and $v_i\in V, 1\leq i\leq5$, by \eqref{eq:cl}, \eqref{eq:COM-Representation-1}-\eqref{semi-direct-com-3-Lie-2}, we have
  \begin{eqnarray*}
  &&[\{x_1+v_1,x_2+v_2,x_3+v_3\}_{\g\oplus V},x_4+v_4,x_5+v_5]_{\g\oplus V}\\&&
  +[x_3+v_3,\{x_1+v_1,x_2+v_2,x_4+v_4\}_{\g\oplus V},x_5+v_5]_{\g\oplus V}\\&&
  +[x_3+v_3,x_4+v_4,\{x_1+v_1,x_2+v_2,x_5+v_5\}_{\g\oplus V}]_{\g\oplus V}\\
   &&+\{[x_1+v_1,x_2+v_2,x_3+v_3]_{\g\oplus V},x_4+v_4,x_5+v_5\}_{\g\oplus V}\\&&
   +\{x_3+v_3,[x_1+v_1,x_2+v_2,x_4+v_4]_{\g\oplus V},x_5+v_5\}_{\g\oplus V}\\
   &&+\{x_3+v_3,x_4+v_4,[x_1+v_1,x_2+v_2,x_5+v_5]_{\g\oplus V}\}_{\g\oplus V}\\&&-[x_1+v_1,x_2+v_2,\{x_3+v_3,x_4+v_4,x_5+v_5\}_{\g\oplus V}]_{\g\oplus V}\\
   &&-\{x_1+v_1,x_2+v_2,[x_3+v_3,x_4+v_4,x_5+v_5]_{\g\oplus V}\}_{\g\oplus V}\\
   &=&[\{x_1,x_2,x_3\},x_4,x_5]+[x_3,\{x_1,x_2,x_4\},x_5]+[x_3,x_4,\{x_1,x_2,x_5\}]\\
   &&+\{[x_1,x_2,x_3],x_4,x_5\}+\{x_3,[x_1,x_2,x_4,x_5\}+\{x_3,x_4,[x_1,x_2,x_5]\}\\
   &&-[x_1,x_2,\{x_3,x_4,x_5\}]-\{x_1,x_2,[x_3,x_4,x_5]\}\\
   &&+\Big(\rho(\{x_1,x_2,x_3\},x_4)+\mu([x_1,x_2,x_3],x_4)+\rho(x_3,\{x_1,x_2,x_4\})\\
   &&+\mu(x_3,[x_1,x_2,x_4])-[\rho(x_1,x_2),\mu(x_3,x_4)]+[\rho(x_3,x_4),\mu(x_1,x_2)]\Big)v_5\\
   &&+\Big(\rho(\{x_1,x_2,x_3\},x_5)+\rho(x_4,\{x_1,x_2,x_5\})+\mu([x_1,x_2,x_4],x_5)\\
   &&+\mu(x_4,[x_1,x_2,x_5])-[\rho(x_1,x_2),\mu(x_4,x_5)]+[\rho(x_4,x_5),\mu(x_1,x_2)]\Big)v_3\\
   &&+\Big(\rho(x_5,\{x_1,x_2,x_3\})+\mu(x_5,[x_1,x_2,x_3])+\rho(\{x_1,x_2,x_5\},x_3)\\
   &&+\mu([x_1,x_2,x_5],x_3)-[\rho(x_1,x_2),\mu(x_5,x_3)]+[\rho(x_5,x_3),\mu(x_1,x_2)]\Big)v_4\\
   &&+\Big(\rho(x_4,x_5)\mu(x_2,x_3)+\rho(x_5,x_3)\mu(x_2,x_4)+\rho(x_3,x_4)\mu(x_2,x_5)-\rho(x_2,\{x_3,x_4,x_5\})\\
   &&+\mu(x_4,x_5)\rho(x_2,x_3)+\mu(x_5,x_3)\rho(x_2,x_4)+\mu(x_3,x_4)\rho(x_2,x_5)-\mu(x_2,[x_3,x_4,x_5])\Big)v_1\\
   &&+\Big(\rho(x_4,x_5)\mu(x_3,x_1)+\rho(x_5,x_3)\mu(x_4,x_2)+\rho(x_3,x_4)\mu(x_5,x_1)-\rho(\{x_3,x_4,x_5\},x_1)\\
   &&+\mu(x_4,x_5)\rho(x_3,x_1)+\mu(x_5,x_3)\rho(x_4,x_2)+\mu(x_3,x_4)\rho(x_5,x_1)-\mu([x_3,x_4,x_5],x_1)\Big)v_2\\
   &=&0.
  \end{eqnarray*}
Thus $(\g\oplus V, [\cdot,\cdot,\cdot]_{\g\oplus V},\{\cdot,\cdot,\cdot\}_{\g\oplus V})$ is a compatible $3$-Lie algebra if and only if $(V;\rho,\mu)$ is a representation of $\g$.
 \end{proof}

Let $(V,\rho,\mu)$ be a representation of a compatible $3$-Lie algebra $(\g,[\cdot,\cdot,\cdot],\{\cdot,\cdot,\cdot\})$. Define $\rho^*, \mu^*:\wedge^2\g\rightarrow\gl(V^*)$ by
\begin{eqnarray*}
\langle\rho^*(x_1,x_2)\alpha,v\rangle&=&-\langle\alpha,\rho(x_1,x_2)v\rangle,\\
\langle\mu^*(x_1,x_2)\alpha,v\rangle&=&-\langle\alpha,\mu(x_1,x_2)v\rangle,\quad \forall \alpha\in V^*,x_1,x_2\in \g,v\in V.
 \end{eqnarray*}
\begin{pro}
With the above notations, $(V^*,\rho^*,\mu^*)$ is a representation of $\g$, called the dual representation.
\end{pro}
\begin{proof}
It is straightforward.
\end{proof}

\begin{ex}{\rm Let $(\g,[\cdot,\cdot,\cdot],\{\cdot,\cdot,\cdot\})$ be a compatible $3$-Lie algebra. For $x,y\in \g,$ define $\ad,\AD:\wedge^2\g\lon\gl(\g)$ by
\begin{eqnarray}\label{com-adjoint-rep}
\ad_{x,y}z=[x,y,z] \quad \mbox{and} \quad\AD_{x,y}z=\{x,y,z\} \quad
\forall z\in \g.
\end{eqnarray}
 Then $(\g,\ad,\AD)$ is a representation of $(\mathfrak{g},[\cdot,\cdot,\cdot],\{\cdot,\cdot,\cdot\})$, which is called the {\bf adjoint representation}. The dual representation $(\g,\ad^*,\AD^*)$ of the adjoint representation $(\g,\ad,\AD)$ of a compatible $3$-Lie algebra $\g$ is called the {\bf coadjoint representation}.}
 \end{ex}

Let $(V;\rho,\mu)$  be a representation of a compatible $3$-Lie algebra $(\g,[\cdot,\cdot,\cdot],\{\cdot,\cdot,\cdot\}).$
For convenience, we use $\pi_1, \pi_2 :\wedge^3\g\rightarrow\g$ to indicate the $3$-Lie bracket $[\cdot,\cdot,\cdot]$ and $\{\cdot,\cdot,\cdot\}$ respectively.

\begin{pro}
 With the above notations,  the pair $(\hat{\pi}_1+\hat{\rho},\hat{\pi}_2+\hat{\mu})$ is a Maurer-Cartan element of the bidifferential graded Lie algebra $(C^*(\g\oplus V,\g\oplus V),[\cdot,\cdot]_\Li,{\rm d}_1=0,{\rm d}_2=0)$, i.e.,
\begin{eqnarray}\label{eq:Maurer-Cartatn 12}
  [\hat{\pi}_1+\hat{\rho},\hat{\pi}_1+\hat{\rho}]_{\Li}=0,\quad[\hat{\pi}_2+\hat{\mu},\hat{\pi}_2+\hat{\mu}]_{\Li}=0,\quad [\hat{\pi}_1+\hat{\rho},\hat{\pi}_2+\hat{\mu}]_{\Li}=0.
\end{eqnarray}
\end{pro}
\begin{proof}
  By Lemma \ref{semi-com-3-Lie} and Proposition \ref{3-Lie-graded}, since $(\g\oplus V, [\cdot,\cdot,\cdot]_{\g\oplus V},\{\cdot,\cdot,\cdot\}_{\g\oplus V})$ is a compatible $3$-Lie algebra, thus we have $[\hat{\pi}_1+\hat{\rho},\hat{\pi}_1+\hat{\rho}]_{\Li}=0,$
    $[\hat{\pi}_2+\hat{\mu},\hat{\pi}_2+\hat{\mu}]_{\Li}=0$
 and the compatibility condition between $[\cdot,\cdot,\cdot]_{\g\oplus V}$ and $\{\cdot,\cdot,\cdot\}_{\g\oplus V}$ is equivalent to  $[\hat{\pi}_1+\hat{\rho},\hat{\pi}_2+\hat{\mu}]_{\Li}=0.$
 Thus $(\hat{\pi}_1+\hat{\rho},\hat{\pi}_2+\hat{\mu})$ is a Maurer-Cartan element of the bidifferential graded Lie algebra $(C^*(\g\oplus V,\g\oplus V),[\cdot,\cdot]_\Li,{\rm d}_1=0,{\rm d}_2=0)$.
\end{proof}

Define the space of $n$-cochains $\huaC_{\cC}^n(\g,V)$, $(n\ge 1)$ by
$$\huaC_{\cC}^n(\g,V)=\underbrace{\huaC_{\Li}^{n}(\g,V)\oplus \huaC_{\Li}^{n}(\g,V)\cdots\oplus \huaC_{\Li}^{n}(\g,V)}_{\mbox{$n$~{copies}}}.$$

Define $\delta^1:\huaC_{\cC}^{1}(\g,V)\longrightarrow \huaC_{\cC}^{2}(\g,V)$  by
$$\delta^1 f=(\dM^1_{\pi_1+\rho} f,\dM^1_{\pi_2+\mu} f),\quad f\in \Hom(\g,V)$$
and
define $\delta^n:\huaC_{\cC}^{n}(\g,V)\longrightarrow \huaC_{\cC}^{n+1}(\g,V)$ for $n>1$ by
\begin{eqnarray*}
    \delta^n(\omega_1,\cdots,\omega_{n})=(\dM^n_{\pi_1+\rho}\omega_1,\cdots,\underbrace{\dM^n_{\pi_2+\mu}\omega_{i-1}+\dM^n_{\pi_1+\rho}\omega_i}_i,\cdots,\dM^n_{\pi_2+\mu}\omega_{n}),
\end{eqnarray*}
where $(\omega_1,\omega_2,\cdots,\omega_n)\in\huaC_{\cC}^n(\g,V)$ and $2\leq i\leq n$.

 By the graded Jacobi identity, \eqref{eq:Maurer-Cartatn 12} is also equivalent to
\begin{equation}\label{eq:MC-direct sum}
\dM^{n+1}_{\pi_1+\rho}\circ\dM^n_{\pi_1+\rho}=0,\quad \dM^{n+1}_{\pi_2+\mu}\circ\dM^n_{\pi_2+\mu}=0,\quad \dM^{n+1}_{\pi_1+\rho}\circ \dM^n_{\pi_2+\mu}+\dM^{n+1}_{\pi_2+\mu}\circ \dM^n_{\pi_1+\rho}=0.
\end{equation}

Similar to the proof of Theorem \ref{thm:cohomology of CLA} or using the relation \eqref{eq:MC-direct sum}, we also have $\delta^{n+1}\circ\delta^n=0$. Thus $(\huaC_{\cC}^*(\g,V)=\oplus_{n=1}^\infty \huaC_{\cC}^n(\g,V),\delta^*)$ is a cochain complex.

\begin{defi}Let $(\mathfrak{g},[\cdot,\cdot,\cdot],\{\cdot,\cdot,\cdot\})$ be a compatible $3$-Lie algebra and $(V;\rho,\mu)$  a
representation.
  The cohomology of the cochain complex $(\huaC_{\cC}^*(\g,V),\delta^*)$  is  called {\bf the cohomology of  $(\g,[\cdot,\cdot,\cdot],\{\cdot,\cdot,\cdot\})$ with coefficients in $V$}. Denote the $n$-th cohomology group by $\huaH^n(\g;V).$ 

\end{defi}

\subsection{Abelian extensions of compatible $3$-Lie algebras}

In this subsection, we study abelian extensions of a compatible $3$-Lie
algebra and show that they are classified by the second
cohomology group.
\begin{defi}\label{defi:isomorphic}
\begin{itemize}
\item[\rm(1)] Let $(\g,[\cdot,\cdot,\cdot]_\g,\{\cdot,\cdot,\cdot\}_\g)$, $(V,[\cdot,\cdot,\cdot]_V,\{\cdot,\cdot,\cdot\}_V)$, $(\hat{\g},[\cdot,\cdot,\cdot]_{\hat{\g}},\{\cdot,\cdot,\cdot\}_{\hat{\g}})$ be compatible $3$-Lie algebras. An extension of compatible $3$-Lie algebras is a short exact sequence of compatible $3$-Lie algebras:
\begin{eqnarray*}
0\longrightarrow V\stackrel{\id}{\longrightarrow}\hat{\g}\stackrel{\p}\longrightarrow\g\longrightarrow0.
\end{eqnarray*}
We say that $\hat{\g}$ is an  {\bf extension} of $\g$ by $V$. An
extension of $\g$ by $V$ is called {\bf abelian}  if  \begin{equation}\label{eq:ac} [u,v,\alpha]_{\hat{\g}}=\{u,v,\alpha\}_{\hat{\g}}=0,\quad \forall  u,v\in V, \alpha\in\hat{\g}.\end{equation}
Here we identify $V$ with its image $\id(V)$, and omit the notation $\id$ for convenience.
\item[\rm(2)]
A {\bf linear section} of $\hat{\g}$
 is a linear map  $\sigma:\g\rightarrow\hat{\g}$  satisfying $\p\circ \sigma=\Id_{\g}$.
\item[\rm(3)] Two abelian extensions $\hat{\g}_1$ and $\hat{\g}_2$ of $\g$
by $V$   are said to be {\bf isomorphic} if there exists a
compatible $3$-Lie algebra   isomorphism
$\theta:\hat{\g}_1\longrightarrow \hat{\g}_2$ such that we have
the following commutative diagram:
\begin{equation}\label{diagram1}
\begin{array}{ccccccccc}
0&\longrightarrow& V&\stackrel{\id_1}\longrightarrow&\hat{\g}_1&\stackrel{\p_1}\longrightarrow&\g&\longrightarrow&0\\
 &            &\Big\|&       &\theta\Big\downarrow&          &\Big\|& &\\
 0&\longrightarrow&V&\stackrel{\id_2}\longrightarrow&\hat{\g}_2&\stackrel{\p_2}\longrightarrow&\g&\longrightarrow&0.
 \end{array}\end{equation}
\end{itemize}
\end{defi}

Let $\hat{\g}$ be an abelian extension of $\g$ by $V$, and
$\sigma:\g\rightarrow\hat{\g}$ a linear section. Define linear maps $\rho,\mu:\wedge^2\g\longrightarrow\gl(V)$ by
\begin{eqnarray}
 \label{eq:res3} \rho(x,y)u&=&[\sigma(x),\sigma(y),u]_{\hat{\g}},\\
 \label{eq:res4} \mu(x,y)u&=&\{\sigma(x),\sigma(y),u\}_{\hat{\g}},
\end{eqnarray}
for all $x,y,\in\g,~u\in V$.
\begin{pro}
  With the above notations, $(V;\rho,\mu)$ is a representation of the compatible $3$-Lie algebra $(\g,[\cdot,\cdot,\cdot]_\g,\{\cdot,\cdot,\cdot\}_\g)$, and does not depend on the choice of linear sections $\sigma$.
  Moreover, isomorphic abelian extensions give the same representation of $\g$ on $V.$
\end{pro}
\begin{proof}
For all $x_1, x_2, x_3, x_4\in\g,u\in V,$ by \eqref{eq:jacobi1}, \eqref{representation-1},  \eqref{representation-2}   and the abelian condition \eqref{eq:ac}, we have
\begin{eqnarray*}
&&\rho(x_1,x_2)\rho(x_3,x_4)u-\rho(x_3,x_4)\rho(x_1,x_2)u-\rho([x_1,x_2,x_3]_{\g},x_4)u-\rho(x_3,[x_1,x_2,x_4]_{\g})u\\
&=&[\sigma(x_1),\sigma(x_2),[\sigma(x_3),\sigma(x_4),u]_{\hat{\g}}]_{\hat{\g}}-[\sigma(x_3),\sigma(x_4),[\sigma(x_1),\sigma(x_2),u]_{\hat{\g}}]_{\hat{\g}}\\
&&-[\sigma[x_1,x_2,x_3]_{\g},\sigma(x_4),u]_{\hat{\g}}-[\sigma(x_3),\sigma[x_1,x_2,x_4]_{\g},u]_{\hat{\g}}\\
&=&[\sigma(x_1),\sigma(x_2),[\sigma(x_3),\sigma(x_4),u]_{\hat{\g}}]_{\hat{\g}}-[\sigma(x_3),\sigma(x_4),[\sigma(x_1),\sigma(x_2),u]_{\hat{\g}}]_{\hat{\g}}\\
&&-[[\sigma(x_1),\sigma(x_2),\sigma(x_3)]_{\hat{\g}}+\underbrace{\sigma[x_1,x_2,x_3]_{\g}-[\sigma(x_1),\sigma(x_2),\sigma(x_3)]_{\hat{\g}}}_{\in V},\sigma(x_4),u]_{\hat{\g}}\\
&&-[\sigma(x_3),[\sigma(x_1),\sigma(x_2),\sigma(x_4)]_{\hat{\g}}+\underbrace{\sigma[x_1,x_2,x_4]_{\g}-[\sigma(x_1),\sigma(x_2),\sigma(x_4)]_{\hat{\g}}}_{\in V},u]_{\hat{\g}}\\
&=&0.
\end{eqnarray*}
Similarly, we have
\begin{eqnarray*}
\rho(x_{1},[x_{2},x_{3},x_{4}]_\g)u-\rho(x_{3},x_{4})\rho(x_{1},x_{2})u-\rho(x_{2},x_{4})\rho(x_{1},x_{3})u
+\rho(x_{2},x_{3})\rho(x_{1},x_{4})u=0.
\end{eqnarray*}
Thus, $\rho$ is a representation of the $3$-Lie algebra of $(\g,[\cdot,\cdot,\cdot]_\g).$
By similar calculations, we can deduce that $\mu$ is a representation of the $3$-Lie algebra of $(\g,\{\cdot,\cdot,\cdot\}_\g)$ and \eqref{eq:COM-Representation-1}-\eqref{eq:COM-Representation-2} hold.  Therefore, $(V;\rho,\mu)$ is a representation of the compatible $3$-Lie algebra $(\g,[\cdot,\cdot,\cdot]_\g,\{\cdot,\cdot,\cdot\}_\g).$

Let $\sigma':\g\rightarrow\hat{\g}$ be another section, and $(V;\rho',\mu')$ be the corresponding representation of the compatible $3$-Lie algebra $(\g,[\cdot,\cdot,\cdot]_{\g},\{\cdot,\cdot,\cdot\}_{\g}).$
Since $\p(\sigma(x)-\sigma'(x))=x-x=0,$ it follows that $\sigma(x)-\sigma'(x)\in \Ker(\p)\cong V.$
Also by the abelian condition \eqref{eq:ac}, we have
\begin{eqnarray*}
{}(\rho'(x_1,x_2)-\rho(x_1,x_2))u&=&[\sigma'(x_1),\sigma'(x_2),u]_{\hat{\g}}-[\sigma(x_1),\sigma(x_2),u]_{\hat{\g}}=0,\\
{}(\mu'(x_1,x_2)-\mu(x_1,x_2))u&=&\{\sigma'(x_1),\sigma'(x_2),u\}_{\hat{\g}}-\{\sigma(x_1),\sigma(x_2),u\}_{\hat{\g}}=0,
\end{eqnarray*}
which implies that the representation $(V;\rho,\mu)$ of the compatible $3$-Lie algebra $\g$ does not depend on the choice of linear sections.

Suppose that $\hat{\g}_1$ and $\hat{\g}_2$ are isomorphic abelian extensions, and $\theta:\hat{\g}_1\longrightarrow \hat{\g}_2$ is the
compatible $3$-Lie algebra homomorphism satisfying $\theta\circ \id_1=\id_2, \p_1=\p_2\circ \theta.$ Choose linear sections $\sigma$ and $\sigma'$ of $\p_1$ and $\p_2$ respectively, we get $\p_2\theta\sigma(x)=\p_1\sigma(x)=x=\p_2\sigma'(x),$ then $\theta\sigma(x)-\sigma'(x)\in \Ker(\p_2)\cong V.$
Thus,   by the abelian condition \eqref{eq:ac}, we have
\begin{eqnarray*}
[\sigma'(x_1),\sigma'(x_2),u]_{\hat{\g}_2}&=&[\theta\sigma(x_1),\theta\sigma(x_2),u]_{\hat{\g}_2}=\theta[\sigma(x_1),\sigma(x_2),u]_{\hat{\g}_1}=[\sigma(x_1),\sigma(x_2),u]_{\hat{\g}_1};\\
\{\sigma'(x_1),\sigma'(x_2),u\}_{\hat{\g}_2}&=&\{\theta\sigma(x_1),\theta\sigma(x_2),u\}_{\hat{\g}_2}=\theta\{\sigma(x_1),\sigma(x_2),u\}_{\hat{\g}_1}=\{\sigma(x_1),\sigma(x_2),u\}_{\hat{\g}_1}.
\end{eqnarray*}
Therefore, isomorphic abelian extensions give the same $\rho$ and $\mu$.
The proof is finished.
\end{proof}

Let $\sigma:\g\rightarrow \hat{\g}$ be a linear section of the abelian extension. We define $\omega_1,\omega_2\in \Hom(\wedge^3\g, V)$ by
\begin{eqnarray}
  \label{eq:str1}\omega_1(x,y,z)&=&[\sigma(x),\sigma(y),\sigma(z)]_{\hat{\g}}-\sigma[x,y,z]_{\g},\\
  \label{eq:str2}\omega_2(x,y,z)&=&\{\sigma(x),\sigma(y),\sigma(z)\}_{\hat{\g}}-\sigma\{x,y,z\}_{\g},
\quad \forall x,y,z\in\g.
\end{eqnarray}

\begin{pro}\label{cohomological-1}
With the above notations,  $(\omega_1,\omega_2)\in\huaC_{\cC}^{2}(\g,V)$ is a $2$-cocycle of the compatible $3$-Lie algebra $(\g,[\cdot,\cdot,\cdot]_\g,\{\cdot,\cdot,\cdot\}_\g)$ with coefficients in $V$,  where the representation $\rho$ and $\mu$ are given by \eqref{eq:res3} and \eqref{eq:res4}. Moreover, its cohomological
class does not depend on the choice of sections.

\end{pro}
\begin{proof}
Let $(\mathfrak{g},[\cdot,\cdot,\cdot]_{\g},\{\cdot,\cdot,\cdot\}_{\g})$ be a compatible $3$-Lie algebra with $\pi_1(x_1,x_2,x_3)=[x_1,x_2,x_3]_{\g}$, $\pi_2(x_1,x_2,x_3)=\{x_1,x_2,x_3\}_{\g}.$
Since $(\g,[\cdot,\cdot,\cdot]_{\hat{\g}},\{\cdot,\cdot,\cdot\}_{\hat{\g}})$ is a compatible $3$-Lie algebra,
we have
\begin{eqnarray*}
&&[\sigma(x_1),\sigma(x_2),[\sigma(x_3),\sigma(x_4),\sigma(x_5)]_{\hat{\g}_1}]_{\hat{\g}_1}\\
&=&[[\sigma(x_1),\sigma(x_2),\sigma(x_3)]_{\hat{\g}_1},\sigma(x_4),\sigma(x_5)]_{\hat{\g}_1}+[\sigma(x_3),[\sigma(x_1),\sigma(x_2),\sigma(x_4)]_{\hat{\g}_1},\sigma(x_5)]_{\hat{\g}_1}\\
&&+[\sigma(x_3),\sigma(x_4),[\sigma(x_1),\sigma(x_2),\sigma(x_5)]_{\hat{\g}_1}]_{\hat{\g}_1}.
\end{eqnarray*}
By \eqref{eq:res3}-\eqref{eq:str2}, we get
\begin{eqnarray*}
 &&\omega_1(x_1,x_2,[x_3,x_4,x_5]_{\g})-\omega_1([x_1,x_2,x_3]_{\g},x_4,x_5)-\omega_1(x_3,[x_1,x_2,x_4]_{\g},x_5)\\
\nonumber&&-\omega_1(x_3,x_4,[x_1,x_2,x_5]_{\g})+\rho(x_1,x_2)\omega_1(x_3,x_4,x_5)-\rho(x_3,x_4)\omega_1(x_1,x_2,x_5)\\
\nonumber&&-\rho(x_4,x_5)\omega_1(x_1,x_2,x_3)-\rho(x_5,x_3)\omega_1(x_1,x_2,x_4)=0,
\end{eqnarray*}
which implies that $\dM^2_{\pi_1+\rho}\omega_1=0.$ Similarly, we can deduce that
$$\dM^2_{\pi_2+\mu}\omega_1+\dM^2_{\pi_1+\rho}\omega_2=0\quad \mbox{and}\quad\dM^2_{\pi_2+\mu}\omega_2=0.$$
Then we have $\delta^2(\omega_1,\omega_2)=0,$ i.e.  $(\omega_1,\omega_2)\in\huaC_{\cC}^{2}(\g,V)$ is a $2$-cocycle of the compatible $3$-Lie algebra $(\g,[\cdot,\cdot,\cdot]_\g,\{\cdot,\cdot,\cdot\}_\g)$ with coefficients in $(V;\rho,\mu).$

Let $\sigma':\g\rightarrow\hat{\g}$ be another section of the abelian extension and $(\omega_1',\omega_2')$ be the associated $2$-cocycle.
Assume that $\sigma'=\sigma+\tau$ for $\tau\in \Hom(\g,V)$. Then we have
\begin{eqnarray*}
  (\omega_1'-\omega_1)(x,y,z)
  &=&[\sigma'(x),\sigma'(y),\sigma'(z)]_{\hat{\g}}-\sigma'
  [x,y,z]_{\g}-[\sigma(x),\sigma(y),\sigma(z)]_{\hat{\g}}+\sigma[x,y,z]_{\g}\\
  &=&[\sigma(x),\sigma(y),\tau(z)]_{\hat{\g}}+[\tau(x),\sigma(y),\sigma(z)]_{\hat{\g}}+[\sigma(x),\tau(y),\sigma(z)]_{\hat{\g}}-\tau([x,y,z]_{\g})\\
  &=&\rho(x,y)\tau(z)+\rho(y,z)\tau(x)+\rho(z,x)\tau(y)-\tau([x,y,z]_{\g})\\
  &=&\dM^1_{\pi_1+\rho}\tau(x,y,z),\\
  (\omega_2'-\omega_2)(x,y,z)&=&\{\sigma'(x),\sigma'(y),\sigma'(z)\}_{\hat{\g}}-\sigma'\{x,y,z\}_{\g}-
  \{\sigma(x),\sigma(y),\sigma(z)\}_{\hat{\g}}+\sigma\{x,y,z\}_{\g}\\
&=&\{\sigma(x),\sigma(y),\tau(z)\}_{\hat{\g}}+\{\tau(x),\sigma(y),\sigma(z)\}_{\hat{\g}}+\{\sigma(x),\tau(y),\sigma(z)\}_{\hat{\g}}-\tau(\{x,y,z\}_{\g})\\
  &=&\mu(x,y)\tau(z)+\mu(y,z)\tau(x)+\mu(z,x)\tau(y)-\tau([x,y,z]_{\g})\\
&=&\dM^1_{\pi_2+\mu}\tau(x,y,z),
\end{eqnarray*}
which implies that $(\omega_1',\omega_2')-(\omega_1,\omega_2)=\delta^1\tau$. Thus $(\omega_1',\omega_2')$ and $(\omega_1,\omega_2)$ are in the same cohomology class.
\end{proof}

By choosing a linear section $\sigma:\g\rightarrow \hat{\g}$,      we can transfer the compatible $3$-Lie algebra structure on $\hat{\g}$ to that on
$\g\oplus V$, for which we denote by $[\cdot,\cdot,\cdot]_{\rho,\omega_1}$ and $\{\cdot,\cdot,\cdot\}_{\mu,\omega_2}$, where $\rho,\mu,\omega_1,\omega_2$ are given by \eqref{eq:res3}-\eqref{eq:str2} respectively. More precisely, $[\cdot,\cdot,\cdot]_{\rho,\omega_1}$ and $\{\cdot,\cdot,\cdot\}_{\mu,\omega_2}$ are given by
\begin{eqnarray}
\label{eq:w-bracket1}
[x+u,y+v,z+w]_{\rho,\omega_1}&=&[x,y,z]_{\g}+\rho(x,y)w+\rho(y,z)u+\rho(z,x)v+\omega_1(x,y,z),\\
\label{eq:w-bracket2}\{x+u,y+v,z+w\}_{\mu,\omega_2}&=&\{x,y,z\}_{\g}+\mu(x,y)w+\mu(y,z)u+\mu(z,x)v+\omega_2(x,y,z),
\end{eqnarray}
for $x,y,z\in \g, u,v,w\in V.$

In the sequel, we only consider abelian extensions of the form $$\g\oplus_{\rho,\mu,\omega_1,\omega_2} V:=(\g\oplus V,[\cdot,\cdot,\cdot]_{\rho,\omega_1},\{\cdot,\cdot,\cdot\}_{\mu,\omega_2}).$$

Now we are ready to classify abelian extensions of compatible 3-Lie algebras using the second cohomology group.

\begin{thm}
Two abelian extensions of $3$-Lie algebras $\g\oplus_{\rho,\mu,\omega_1,\omega_2} V$ and $ \g\oplus_{\rho,\mu,\omega_1',\omega_2'} V$ are isomorphic if and only if~~$2$-cocycles $(\omega_1,\omega_2)$ and $(\omega_1',\omega_2')$  are in the same cohomology class.
\end{thm}
\begin{proof}
Let $\g\oplus_{\rho,\mu,\omega_1,\omega_2} V$ and $ \g\oplus_{\rho,\mu,\omega_1',\omega_2'} V$ be isomorphic abelian extensions,
 and $\theta:\g\oplus_{\rho,\mu,\omega_1,\omega_2}V\rightarrow \g\oplus_{\rho,\mu,\omega_1',\omega_2'} V$ be the corresponding isomorphism.
 Then there exist $\tau:\g\rightarrow V$ such that
$$\theta(x+u)=x+\tau(x)+u,\quad \forall x\in\g,u\in V.$$
For all $x,y,z\in\g,$ we have
\begin{eqnarray*}
\theta[x,y,z]_{\rho,\omega_1}
=\theta([x,y,z]_{\g}+\omega_1(x,y,z))
=[x,y,z]_{\g}+\tau([x,y,z]_{\g})+\omega_1(x,y,z),
\end{eqnarray*}
and
\begin{eqnarray*}
[\theta(x),\theta(y),\theta(z)]_{\rho,\omega_1'}
&=&[x+\tau(x),y+\tau(y),z+\tau(z)]_{\rho,\omega_1'}\\
&=&[x,y,z]_{\g}+\omega_1'(x,y,z)+\rho(x,y)\tau(z)+\rho(y,z)\tau(x)+\rho(z,x)\tau(y).
\end{eqnarray*}
Therefore, by $ \theta[x,y,z]_{\rho,\omega_1}=[\theta(x),\theta(y),\theta(z)]_{\rho,\omega_1'}$, we get
\begin{eqnarray*}
(\omega_1-\omega_1')(x,y,z)=\rho(x,y)\tau(z)+\rho(y,z)\tau(x)+\rho(z,x)\tau(y)-\tau([x,y,z]_{\g})=\dM^1_{\pi_1+\rho}\tau(x,y,z),
\end{eqnarray*}

Similarly, we have
\begin{eqnarray*}
(\omega_2-\omega_2')(x,y,z)=\mu(x,y)\tau(z)+\mu(y,z)\tau(x)+\mu(z,x)\tau(y)-\tau(\{x,y,z\}_{\g})=\dM^1_{\pi_2+\mu}\tau(x,y,z),
\end{eqnarray*}
which means that $(\omega_1-\omega_1',\omega_2-\omega_2')=\delta^1\tau$
Therefore, $(\omega_1,\omega_2)$ and $(\omega_1',\omega_2')$ are in the same cohomology class.

The converse part can be proved similarly, and we omit details.
\end{proof}

 \end{document}